\documentclass[smallextended,referee,envcountsect]{svjour3}
\smartqed

\usepackage{amsmath,amssymb,amsfonts,amsopn}
\usepackage{graphicx}
\usepackage{xcolor}
\usepackage{booktabs,tabularx}
\usepackage{url}
\usepackage{hyperref}
\hypersetup{hidelinks,pdftitle={Trajectory-Wise Certification for Vector-Field Mirror Descent with Deterministic Finite Differences},pdfauthor={Masahito Hayashi}}

\newcommand{\Label}[1]{\label{#1}}

\journalname{Journal of Optimization Theory and Applications}

\begin{document}

\title{Trajectory-Wise Certification for Vector-Field Mirror Descent with Deterministic Finite Differences}

\author{Masahito Hayashi}

\institute{Masahito Hayashi \at
School of Data Science, The Chinese University of Hong Kong, Shenzhen, Longgang District, Shenzhen 518172, China \\
International Quantum Academy, Futian District, Shenzhen 518048, China \\
Graduate School of Mathematics, Nagoya University, Furo-cho, Chikusa-ku, Nagoya 464-8602, Japan \\
\email{hmasahito@cuhk.edu.cn}}

\date{Received: date / Accepted: date}

\maketitle

\begin{abstract}
We study mirror descent in which the objective gradient is replaced by a general vector field. Since such a field does not automatically relate the mirror update to the objective gap, we introduce a trajectory-wise generalized relative-smoothness condition and a generalized star-convexity interface. 
Together they yield a finite-horizon, a posteriori last-iterate certificate with an accumulated-stepsize term and an exceptional-region-dependent error term. 
We also give a pointwise sufficient condition for positive admissible stepsizes, displaying the effects of objective smoothness, mirror-map conditioning, and vector-field mismatch. 
We then construct a deterministic zeroth-order instance from coordinate central differences. 
The function values place the unknown gradient in an explicit uncertainty box, and verification of the interface reduces to robust conic dominance. 
We derive an explicit scaling of the central-difference vector that guarantees the required dominance over the gradient uncertainty set. 
Under uniform Hessian bounds, the resulting vector field satisfies the interface outside an explicit resolution-dependent neighborhood. 
The final guarantee combines a rate term governed by accumulated accepted stepsizes with a finite-resolution error floor.
\end{abstract}


\section{Introduction}
\Label{sec:intro}

\subsection{Problem setting and contributions}

Mirror descent is a basic framework for optimization in non-Euclidean geometry. A differentiable and strictly convex mirror map $\Phi$ transports a primal point $\mathbf{x}$ to the dual coordinate $\nabla\Phi(\mathbf{x})$, where an update direction is applied before mapping back to the primal space. When the update direction is the gradient of the objective, the usual analysis uses $\nabla f$ in two distinct ways: it determines the dual-space displacement, and convexity converts its inner product with $\mathbf{x}-\mathbf{x}_*$ into a function-value gap. Relative smoothness and related Bregman techniques extend this mechanism beyond Euclidean Lipschitz-gradient assumptions \cite{NemirovskiYudin1983,LuFreundNesterov2018,BauschkeBolteTeboulle2017}.

This paper asks what remains of this mechanism when the gradient in the mirror update is replaced by a general vector field $\Omega$. We consider updates of the form
\[
\mathbf{x}_{j+1}=(\nabla\Phi)^{-1}\bigl(\nabla\Phi(\mathbf{x}_j)-\eta_j\Omega(\mathbf{x}_j)\bigr),
\]
whenever the dual update lies in the domain on which $(\nabla\Phi)^{-1}$ is defined. Writing down this update is immediate, but obtaining a function-value guarantee is not. For a general vector field, neither descent nor the usual convexity inequality relating the update direction to the objective gap is automatic. The first purpose of the paper is therefore not to propose an oracle-optimal zeroth-order method, but to identify a verifiable interface under which the Bregman telescoping argument continues to yield a last-iterate guarantee.

The interface has two components. The first is a trajectory-wise generalized relative-smoothness condition comparing the generalized divergence generated by $(f,\Omega)$ with the Bregman divergence generated by $\Phi$ along each realized update. The second is an $\Omega$-dependent star-convexity inequality that converts the vector-field progress term into a function-value gap. We also allow the second inequality to hold only outside a prescribed exceptional region $U$ containing a minimizer. This formulation separates certified progress outside $U$ from the local error term inside it.

Under these two conditions, we derive the finite-horizon a posteriori certificate
\[
f(\mathbf{x}_t)-f(\mathbf{x}_*)
\le
\max\left\{
\frac{D_\Phi(\mathbf{x}_*\|\mathbf{x}_1)}{\sum_{j=1}^{t-1}\eta_j},
\sup_{\mathbf{x}\in U}\bigl(f(\mathbf{x})-f(\mathbf{x}_*)\bigr)
\right\}.
\]
The denominator is the accumulated stepsize of the realized trajectory and is therefore available after the run. The result is primarily a finite-horizon certificate. It yields an asymptotic statement at the exceptional-region-dependent floor only under the additional condition $\sum_{j=1}^{\infty}\eta_j=\infty$. The general framework does not claim that every certificate-driven line-search rule automatically guarantees this nonsummability property.

We complement the conditional certificate with a pointwise feasibility result. Suppose that $f$ is smooth with respect to a primal norm and that the mirror map is strongly convex and smooth with respect to the same norm. At a fixed point, if the relative mismatch between $\Omega$ and $\nabla f$ is sufficiently small, then a nonempty interval of positive stepsizes satisfies the trajectory-wise certificate. This result displays explicitly how objective smoothness, mirror-map conditioning, and oracle mismatch determine the local admissible stepsize scale. It is a pointwise statement and does not by itself provide a uniform positive lower bound along an infinite trajectory.

The second purpose of the paper is to construct a deterministic function-value instance of this general interface. We focus on coordinate central differences in moderate dimension. The two one-sided secant slopes in each coordinate enclose the corresponding component of the true gradient. Consequently, the unknown gradient belongs to an explicit axis-aligned uncertainty box centered at the central-difference vector. A norm approximation alone is not sufficient for the certificate. What is needed is a single scaled proxy vector whose inner product dominates that of every admissible gradient uniformly over a cone of directions aligned with the displacement from the minimizer.

This requirement leads to a robust conic dominance problem. We first solve the corresponding dominance problem exactly for the Euclidean ball containing the uncertainty box. For $d\ge2$, we obtain the exact minimum scaling on that ball. For $d=1$, the exact minimum is different, while the same formula remains a valid sufficient scaling. Applying the ball result to the containing ball gives a dimension-uniform sufficient scaling for the original box. We then combine this construction with a lower bound on the angle between $\nabla f(\mathbf{x})$ and $\mathbf{x}-\mathbf{x}_*$ under uniform Hessian bounds. The resulting finite-difference vector field satisfies the required generalized star-convexity inequality outside an explicit resolution-dependent exceptional set.

For a $\mu$-strongly convex and $L$-smooth objective in Euclidean geometry, the exceptional set is contained in a ball centered at the minimizer whose radius is proportional to $\epsilon\sqrt d$, with explicit dependence on $L/\mu$ and the cone parameter. The finite-difference specialization therefore yields a rate-plus-floor certificate. Before the exceptional set is reached, the certified term decreases according to $\bigl(\sum_{j=1}^{t-1}\eta_j\bigr)^{-1}$; once the exceptional set is reached, the procedure reports attainment of the resolution-dependent target region. In particular, the first term is $O(1/t)$ whenever the accepted stepsizes are uniformly bounded below by a positive constant.

The contributions are summarized as follows.
\begin{enumerate}
\item We formulate mirror descent driven by a general vector field and derive a finite-horizon, trajectory-wise last-iterate certificate from a Bregman comparison inequality and an $\Omega$-dependent star-convexity interface.
\item We give a pointwise sufficient stepsize condition for the trajectory-wise certificate in a general primal-dual norm pair, explicitly displaying the effects of objective smoothness, mirror-map conditioning, and vector-field mismatch.
\item For deterministic coordinate central differences, we identify the gradient uncertainty box, reduce the required inner-product comparison to robust conic dominance, solve the containing-ball scaling problem exactly for $d\ge2$, and transfer the result to a sufficient box scaling valid in every dimension.
\item Combining the conic construction with the Hessian-based angle bound, we obtain an explicit resolution-dependent exceptional neighborhood and a certified last-iterate bound for the stopped finite-difference procedure.
\end{enumerate}

The framework also clarifies the relation with information-geometric iterations. Certain generalized Arimoto--Blahut-type updates can be expressed in the same vector-field mirror form. This connection illustrates that the abstract certificate is not tied to finite differences, although coordinate central differences are the only deterministic zeroth-order instance analyzed in detail here \cite{Arimoto,Blahut,Ramakrishnan,Hayashi2025b}.

\subsection{Relation to existing work}
Classical deterministic derivative-free optimization includes direct-search and model-based trust-region methods; see also the broad review \cite{LarsonMenickellyWild2019}.
Direct-search methods select and test structured direction sets, while model-based methods construct local interpolation or regression models and use acceptance mechanisms to control the step 
\cite{KoldaLewisTorczon2003,AudetDennisMADS2006,AbramsonAudetDennisLeDigabelOrthoMADS2009,ConnScheinbergVicente2009,PowellNEWUOA2004,PowellBOBYQA2009}.
These approaches provide broad frameworks, particularly for stationarity analysis and constrained black-box optimization. The present work addresses a different question. It retains a Bregman mirror update and studies when a realized vector-field trajectory admits a function-value certificate. We do not claim an improvement in worst-case oracle complexity or computational performance over direct-search or trust-region methods.

Randomized zeroth-order optimization commonly constructs one-point or two-point gradient estimators through random perturbations and analyzes expected error, regret, or oracle complexity \cite{FlaxmanKalaiMcMahan2005,NesterovSpokoiny2017,DuchiJordanWainwrightWibisono2014,Shamir2017}. 
The distinction between classical coordinate finite differences and randomized finite-difference approximations has also been discussed from the viewpoint of derivative-free optimization \cite{ScheinbergFiniteDifference2022}.
Our finite-difference specialization is deterministic and uses coordinate function values. 
The distinction relevant here is not a claim of better dimension dependence. Rather, the realized finite-difference information is converted into a deterministic uncertainty set, and the required function-value interface is established through a robust geometric construction.

When exact gradients are available, mirror descent, Bregman proximal methods, and relative smoothness provide established tools for non-Euclidean convergence analysis \cite{NemirovskiYudin1983,LuFreundNesterov2018,BauschkeBolteTeboulle2017}. Recent work has also emphasized certificates based on inequalities verified along an optimization trajectory \cite{DragomirTaylorDAspremontBolte2022}. Our contribution is complementary: the update direction may be a general vector field, and the certificate is paired with an explicit finite-difference construction that manufactures the required function-value interface outside a controlled exceptional neighborhood.

The generalized star-convexity condition is used as an interface condition, not as an attempt to replace the strong-convexity assumptions of the finite-difference specialization. In the abstract theory, it identifies the precise inequality needed to convert vector-field progress into an objective gap. In the finite-difference analysis, ordinary convexity, uniform Hessian bounds, the gradient uncertainty box, and robust conic dominance are combined to verify that interface outside the resolution-dependent exceptional set. This separation permits the general certificate to be stated independently of the mechanism used to establish it.

Finally, the conic part of the analysis concerns robust second-order-cone geometry. The dominance requirement can be expressed as membership of the proxy error in the dual of a circular cone. The exact result proved here concerns the containing Euclidean ball when $d\ge2$; the finite-difference box inherits a sufficient scaling through set containment. Keeping these statements separate is essential because exactness for the ball does not imply exactness for the box \cite{LoboEtAl1998,AlizadehGoldfarb2003}.

\subsection{Organization}

Section~\ref{S2} introduces the vector-field mirror update, the trajectory-wise certificate, and the finite-horizon last-iterate bound. Section~\ref{S3} relates the framework to generalized Arimoto--Blahut-type iterations. Section~\ref{S5:toolbox} collects the norm-pair facts used in the analysis, and Section~\ref{S5} gives the pointwise feasibility condition for the trajectory-wise certificate. Section~\ref{S4} develops the deterministic central-difference construction, the gradient uncertainty box, the robust conic scaling, and the resulting resolution-dependent certificate. The subsequent sections contain the proofs of the general mirror identities, the pointwise feasibility theorem, the exceptional-neighborhood bound, and the Hessian angle estimate.  The proof of the robust dominance result is given in the appendix. Section~\ref{sec:discussion} discusses the finite-horizon interpretation, the role of the generalized star-convexity interface, and the limitations of the present construction.

\section{Generalized Mirror Descent}
\Label{S2}
To generalize mirror descent, we first choose a distance-generating function
$\Phi$.

Throughout the paper, $\Phi$ is assumed to be a Legendre function on
$\operatorname{int}\operatorname{dom}\Phi$, so that
$\nabla\Phi$ is one-to-one and
$(\nabla\Phi)^{-1}=\nabla\Phi^*$ on
$\operatorname{int}\operatorname{dom}\Phi^*$.
Every update uses a stepsize $\eta>0$ such that $ \nabla\Phi(\mathbf{x})-\eta\Omega(\mathbf{x}) \in \operatorname{int}\operatorname{dom}\Phi^*$.  
This condition ensures that $F_{\eta\Omega}(\mathbf{x}) = \nabla\Phi^*\bigl(\nabla\Phi(\mathbf{x})-\eta\Omega(\mathbf{x})\bigr) $ is well defined.
In the full-domain setting used
in the finite-difference specialization, one may take
$\operatorname{dom}\Phi=\mathbb{R}^d$ and impose only the displayed admissibility
condition on the dual update.

Given an objective function $f$,
we consider a generalized mirror descent update. 
For this aim,
instead of the derivative $\nabla f(\mathbf{x}) \in \mathbb{R}^d$,
we focus on a vector valued function
$\Omega(\mathbf{x})\in \mathbb{R}^d$.
The generalized algorithm is given as
\begin{description}
\item[(1)]
Map to dual space: $\boldsymbol{\theta}_j = \nabla \Phi(\mathbf{x}_j)$
\item[(2)]
Gradient update in dual space: $\boldsymbol{\theta}_{j+1} = \boldsymbol{\theta}_j - 
\eta_j \Omega(\mathbf{x}_j)$
\item[(3)]
Map back to primal space: $\mathbf{x}_{j+1} = (\nabla \Phi)^{-1}(\boldsymbol{\theta}_{j+1})$
\end{description}
That is, when we define the function $F_{\eta \Omega}$ as 
$F_{\eta \Omega}(\mathbf{x}):=(\nabla \Phi)^{-1}(\nabla \Phi(\mathbf{x})-\eta \Omega(\mathbf{x}))=\nabla\Phi^*(\nabla\Phi(\mathbf{x})-\eta\Omega(\mathbf{x}))$, the $j+1$-th element
$\mathbf{x}_{j+1}$ is given as $F_{\eta_j \Omega}(\mathbf{x}_{j})$.
At each iteration $j$, we choose a stepsize $\eta_j>0$ by an arbitrary rule, 
possibly depending on the past iterates and observed function values; 
the constant-stepsize choice $\eta_j\equiv \eta$ is included as a special case.

To describe our result, we introduce the 
Bregman divergence $D_{f,\Omega}(\mathbf{x}, \mathbf{y})$ associated with 
a general differential function 
$f$ is defined as:
\begin{align}
D_{f,\Omega}(\mathbf{x}\| \mathbf{y}) 
:= &
\langle \Omega(\mathbf{y}), \mathbf{y} - \mathbf{x} \rangle
-f(\mathbf{y}) + f(\mathbf{x}) .
\end{align}
When $f$ is a convex and differential function, 
its convex conjugate $f^*$ is defined and
$D_{f,\nabla f}(\mathbf{x}\| \mathbf{y}) $ 
is simplified as
$D_{f}(\mathbf{x}\| \mathbf{y}) $, and  
has the form
\begin{align}
D_{f}(\mathbf{x}\| \mathbf{y})= &
\langle 
\nabla f(\mathbf{x})-\nabla f(\mathbf{y}), \mathbf{x} \rangle
-f^*(\nabla f(\mathbf{x})) + f^*(\nabla f(\mathbf{y})) .
\end{align}

Now, given $\eta>0$, 
we introduce the global generalized relative-smoothness condition of 
$f$ and $\Omega$
with respect to $\Phi$ \cite{LuFreundNesterov2018,BauschkeBolteTeboulle2017};
\begin{align}
\eta D_{f,\Omega}(\mathbf{x}\| \mathbf{y}) \le D_\Phi (\mathbf{x}\| \mathbf{y}) .
\Label{RS}
\end{align}

Our results consist of two parts. The first part develops the general theory for mirror descent driven by a vector field under the global generalized relative-smoothness condition.


\begin{theorem}[Monotonicity for a constant stepsize]\Label{Thm1}
Fix a constant stepsize $\eta>0$ and run the generalized mirror-descent update 
\( \mathbf{x}_{j+1}=F_{\eta\Omega}(\mathbf{x}_j)\) for 
$ j=1,\ldots,t-1$. 
If the objective function $f$ satisfies 
the global generalized relative-smoothness condition \eqref{RS}, then 
\( f(\mathbf{x}_{j+1})\le f(\mathbf{x}_j)\)
for \(j=1,\ldots,t-1\).
\end{theorem}

Since the global generalized relative-smoothness condition \eqref{RS}, imposed for all pairs, can be difficult to verify in many applications, we also consider the following trajectory-wise generalized relative-smoothness condition:
\begin{align}
\eta_j D_{f,\Omega}(\mathbf{x}_{j+1}\| \mathbf{x}_{j}) \le 
D_\Phi (\mathbf{x}_{j+1}\| \mathbf{x}_{j}) .
\Label{RS2}
\end{align}
Unlike \eqref{RS}, condition \eqref{RS2} is required only for the realized update pairs $(\mathbf{x}_{j+1},\mathbf{x}_j)$.

To derive an a posteriori guarantee for the achieved objective value, 
we assume that 
$f$ satisfies a generalized condition of the star-convex condition at a minimizer $\mathbf{x}_*$ of $f$ over ${\cal K}$. 
Following \cite{HinderSidfordSohoni2020}, we say that $f$ satisfies 
the generalized star-convex condition
at $\mathbf{x}_*$ on ${\cal K}$ if, for all $\mathbf{x}\in{\cal K}$,
\begin{align}
\langle \Omega (\mathbf{x}), \mathbf{x} - \mathbf{x}_* \rangle
\ge f(\mathbf{x}) - f(\mathbf{x}_*) .
\end{align}
Equivalently, $D_{f,\Omega}(\mathbf{x}_*\|\mathbf{x}) \ge 0$.
This is the special case $\tau=1$-star-convexity used in the recent optimization literature \cite{LezaneLangerKoolen2024}.
This concept is generalized as follows.

\begin{definition}[Generalized star-convexity outside an exceptional region]
Let $\mathbf{x}_*$ be a minimizer of $f$ over $\mathcal K$, and let
$U\subset\mathcal K$ be a prescribed exceptional region containing
$\mathbf{x}_*$.  We say that $f$ is
\emph{generalized star-convex on $\mathcal K\setminus U$ with respect
to $\Omega$} if, for all $\mathbf{x}\in \mathcal K\setminus U$,
\[
\langle \Omega(\mathbf{x}),\, \mathbf{x}-\mathbf{x}_* \rangle \ge f(\mathbf{x})-f(\mathbf{x}_*).
\]
\end{definition}

Then, the following theorem gives 
a posteriori certification of the MD algorithm. 

\begin{theorem}\Label{Thm2}
Let $t\ge2$. Let $\mathbf{x}_*$ be a minimizer of $f$ over $\mathcal K$, 
and let $U\subset\mathcal K$ be a prescribed exceptional region containing $\mathbf{x}_*$.
Assume that $f$ is
generalized star-convex on $\mathcal K\setminus U$ with respect to $\Omega$ and
that the algorithm iterates in $\mathcal K$.  When the obtained points
$\mathbf{x}_1,\ldots,\mathbf{x}_t$ satisfy the trajectory-wise generalized relative-smoothness
condition \eqref{RS2} for $j=1,\ldots,t-1$, we have
\begin{align}
f(\mathbf{x}_t)-f(\mathbf{x}_*)
\le \max\left\{
\frac{D_\Phi(\mathbf{x}_*\|\mathbf{x}_1)}{\sum_{j=1}^{t-1}\eta_j},
\sup_{\mathbf{x}\in U}\bigl(f(\mathbf{x})-f(\mathbf{x}_*)\bigr)
\right\}.
\Label{DJ1}
\end{align}
\end{theorem}

For every fixed $t$, the bound \eqref{DJ1} applies to the computed iterates
$\mathbf{x}_1,\ldots,\mathbf{x}_t$, provided that \eqref{RS2} is verified for
the performed steps $j=1,\ldots,t-1$.  Its first term is evaluated using the
accumulated stepsize $\sum_{j=1}^{t-1}\eta_j$ of those steps.
Theorem~\ref{Thm2} does not impose a particular rule for choosing the
stepsizes and therefore does not by itself guarantee that this accumulated
stepsize diverges as $t\to\infty$.  If, in addition,
\[
\sum_{j=1}^{\infty}\eta_j=\infty,
\]
then the first term in \eqref{DJ1} converges to zero, and hence
\[
\limsup_{t\to\infty}
\bigl(f(\mathbf{x}_t)-f(\mathbf{x}_*)\bigr)
\le
\sup_{\mathbf{x}\in U}
\bigl(f(\mathbf{x})-f(\mathbf{x}_*)\bigr).
\]

In the gradient case $\Omega=\nabla f$, the proof mechanism of Theorem~\ref{Thm2}
reduces to the standard mirror-descent analysis based on the Bregman three-point identity,
which is used to telescope the key inequality (cf.\ the step from (10) to (11) in \cite{FawziLecture11Bregman2023}).
Accordingly, Lemma~\ref{Lem1} can be seen as an $\Omega$-driven analogue of that three-point-identity step.
Our contribution is not the gradient-case algebra itself, but the $\Omega$-driven formulation and the
trajectory-wise (a posteriori) certification condition \eqref{RS2}, which remain meaningful for deterministic
zeroth-order oracles.
Replacing the condition \eqref{RS2} by \eqref{RS},
we obtain a priori certification version as follows.

\begin{corollary}\Label{ThmB2}
Assume that the objective function $f$ satisfies
the global generalized relative-smoothness condition \eqref{RS}.
When $f$ is star-convex at $\mathbf{x}_*$ with a set ${\cal K}$ containing $\mathbf{x}_*$
and 
the algorithm iterates in ${\cal K}$, 
we have \begin{align}
 f(\mathbf{x}_t)- f(\mathbf{x}_*)
\le \frac{D_\Phi(\mathbf{x}_* \| \mathbf{x}_1)}{\eta(t-1)}.\Label{DJ1B}
\end{align}
\end{corollary}
This corollary can be considered as a generalization of 
the convexity case shown in \cite[Theorem 11.1]{FawziLecture11Bregman2023}.

\section{Relation to generalized Arimoto-Blahut algorithm}\Label{S3}
In the area of information theory, 
Arimoto-Blahut algorithm \cite{Arimoto,Blahut} is known.
The recent paper \cite{Ramakrishnan} extended it into a generalized form, which is called 
the generalized Arimoto-Blahut algorithm including quantum systems.
A posteriori certification specialized to \emph{generalized quantum Arimoto--Blahut} (QAB) algorithms is developed in the companion paper~\cite{LiuHayashi2026QABCert} while the present paper focuses on a general vector-field mirror-descent template and its deterministic zeroth-order instantiation.
Later, the paper \cite{Hayashi2025b} extended it for the Bregman divergence system through 
several related studies \cite{Hayashi2025,Hayashi2024}.
To describe the generalized Arimoto-Blahut algorithm under the 
Bregman divergence system,
we consider $d+1$ differential functions
$\Omega^0(\mathbf{x}), \Omega^1(\mathbf{x}), \ldots, \Omega^d(\mathbf{x})$ 
on $\mathbb{R}^d$.
The objective function $f$ is assumed to be 
\begin{align}
f_{\Omega}(\mathbf{x})= \Omega^0(\mathbf{x})+ 
\sum_{j=1}^d x_j \Omega^j(\mathbf{x}).
\end{align}
We set the vector-valued function $\Omega_d(\mathbf{x})$
as $(\Omega^1(\mathbf{x}), \ldots, \Omega^d(\mathbf{x}))$.
Then, we have
\begin{align}
D_{f_\Omega,\Omega_d}(\mathbf{x}\|\mathbf{y})
= \Omega^0(\mathbf{x})- \Omega^0(\mathbf{y})
+\sum_{j=1}^d x_j (\Omega^j(\mathbf{x})- \Omega^j(\mathbf{y})),
\end{align}
which is simplified as $D_{\Omega}(\mathbf{x}\|\mathbf{y})$.
We impose the following condition;
\begin{align}
\eta D_{\Omega}(\mathbf{x}\|\mathbf{y})
\le D_\Phi(\mathbf{x}\|\mathbf{y}). \Label{NMM}
\end{align}
Under the condition \eqref{NMM}
with the Bregman divergence system,
the generalized Arimoto-Blahut algorithm is given as the updating rule
$\mathbf{x}_{t+1}= F_\Omega (\mathbf{x}_t)$,
where $F_\Omega$ is defined as
$F_\Omega(\mathbf{x}):=\nabla \Phi^{-1}(\nabla \Phi(\mathbf{x})-\eta 
\vec{\Omega}(\mathbf{x}))$, and 
$\vec{\Omega} (\mathbf{x})$ is a vector 
$(\Omega^1(\mathbf{x}), \ldots, \Omega^d(\mathbf{x}))$.
We have the following lemma.

\begin{lemma} 
Assume that $\Omega^0,\Omega^1,\ldots,\Omega^d$ are differentiable on an open set and that \eqref{NMM} holds for all pairs in this set. Then, 
the relation
\( \Omega_d(\mathbf{x})=\nabla f_\Omega(\mathbf{x}) \) 
holds at every point $\mathbf{x}$ of the set. 
\end{lemma}

This lemma shows that 
the generalized Arimoto-Blahut algorithm coincides with 
the mirror descent under the condition \eqref{NMM}
while the paper \cite{He} showed the same fact 
for the quantum setting under the convexity of $f_\Omega$.
Since the same results as Theorems \ref{Thm1} and \ref{Thm2}
were shown for the generalized Arimoto-Blahut algorithm 
by \cite{Hayashi2025b},
Theorems \ref{Thm1} and \ref{Thm2}
can be considered as generalizations of the results in \cite{Hayashi2025b}.

\begin{proof}
Since it is sufficient to show the relation
$\frac{\partial }{\partial x_j}  f_\Omega(\mathbf{x})
=\Omega^j(\mathbf{x})$
for $j=1,\ldots, d$.
This relation can be shown from the relation
\begin{align}
\frac{\partial }{\partial x_j} \Omega^0(\mathbf{x})+ 
\sum_{j'=1}^d x_{j'} \frac{\partial }{\partial x_j}\Omega^{j'}(\mathbf{x})=0.
\Label{NM2}
\end{align}
for $j=1,\ldots, d$.
Hence, we will show \eqref{NM2}.

Differentiability implies
\begin{align} 
D_\Omega(\mathbf{x}\| \mathbf{x}\pm\epsilon\mathbf{e}_j) 
=\mp \left( \frac{\partial}{\partial x_j}\Omega^0(\mathbf{x}) + \sum_{j'=1}^d x_{j'} \frac{\partial}{\partial x_j} \Omega^{j'}(\mathbf{x}) \right)\epsilon +o(\epsilon). 
\Label{eq:NMM-first-order-expansion} \end{align} 
Moreover, differentiability of $\Phi$ gives 
\( D_\Phi(\mathbf{x}\| \mathbf{x}\pm\epsilon\mathbf{e}_j) =o(\epsilon). \) 
Applying \eqref{NMM}, dividing by $\epsilon>0$, and letting $\epsilon\downarrow0$, we obtain 
\begin{align} 
\mp\eta \left( \frac{\partial}{\partial x_j}\Omega^0(\mathbf{x}) + \sum_{j'=1}^d x_{j'} \frac{\partial}{\partial x_j} \Omega^{j'}(\mathbf{x}) \right) \le0. 
\end{align} 
Since this inequality holds for both signs, the expression in parentheses is zero. Hence, \eqref{NM2} holds.
\qed\end{proof}

\section{Norm-Pair Preliminaries}\Label{S5:toolbox}
To proceed further analysis, we prepare 
notations for norm-pair.
Let $\|\cdot\|$ be a norm on $\mathbb{R}^d$, and let
\[
\|\mathbf{z}\|_*
:=
\sup_{\|\mathbf{x}\|\le1}
\langle\mathbf{z},\mathbf{x}\rangle
\]
be its dual norm.  We record the norm-pair facts used in the proof of
Theorem~\ref{thm:refined_eta_general_phi}.

For a differentiable function $g$, its Bregman divergence is
\begin{equation}
D_g(\mathbf{x}\|\mathbf{y})
:=
g(\mathbf{x})-g(\mathbf{y})
-\langle\nabla g(\mathbf{y}),\mathbf{x}-\mathbf{y}\rangle.
\Label{eq:Bregman_def_toolbox}
\end{equation}

A differentiable convex function $g$ is $\mu$-strongly convex with respect
to $\|\cdot\|$ if
\[
g(\mathbf{x})
\ge
g(\mathbf{y})
+\langle\nabla g(\mathbf{y}),\mathbf{x}-\mathbf{y}\rangle
+\frac{\mu}{2}\|\mathbf{x}-\mathbf{y}\|^2.
\]
Equivalently,
\begin{equation}
D_g(\mathbf{x}\|\mathbf{y})
\ge
\frac{\mu}{2}\|\mathbf{x}-\mathbf{y}\|^2.
\Label{eq:strong-convexity-bregman-lower}
\end{equation}

A differentiable function $g$ is $L$-smooth with respect to $\|\cdot\|$ if
\begin{equation}
\|\nabla g(\mathbf{x})-\nabla g(\mathbf{y})\|_*
\le
L\|\mathbf{x}-\mathbf{y}\|.
\Label{eq:norm-pair-smoothness}
\end{equation}
The corresponding descent inequality is
\begin{equation}
D_g(\mathbf{x}\|\mathbf{y})
\le
\frac{L}{2}\|\mathbf{x}-\mathbf{y}\|^2.
\Label{eq:Lsmooth_bregman_upper}
\end{equation}

We also use the standard conjugacy relations.  If $\Phi$ is
$\sigma$-strongly convex with respect to $\|\cdot\|$, then $\Phi^*$ is
$(1/\sigma)$-smooth with respect to $\|\cdot\|_*$; in particular,
\begin{equation}
\|\nabla\Phi^*(\mathbf{u})-\nabla\Phi^*(\mathbf{v})\|
\le
\frac{1}{\sigma}\|\mathbf{u}-\mathbf{v}\|_*.
\Label{eq:conjugate_grad_Lip}
\end{equation}
If $\Phi$ is $\beta$-smooth with respect to $\|\cdot\|$, then $\Phi^*$ is
$(1/\beta)$-strongly convex with respect to $\|\cdot\|_*$, and hence
\begin{equation}
D_{\Phi^*}(\mathbf{u}\|\mathbf{v})
\ge
\frac{1}{2\beta}\|\mathbf{u}-\mathbf{v}\|_*^2.
\Label{eq:conjugate_bregman_lower}
\end{equation}

\section{Feasibility of the trajectory-wise generalized relative-smoothness condition}\Label{S5}
This section clarifies how the trajectory-wise generalized relative-smoothness condition \eqref{RS2} can be
\emph{made feasible} and how its admissible stepsizes scale with problem and oracle
parameters.
In our framework, stepsizes may be selected by a certificate-driven procedure
(e.g., backtracking or a short geometric grid search) that directly checks
\eqref{RS2} along the realized iterates.
Therefore, we do not insist on computing an explicit a priori bound on $\eta$
for every application.
Instead, the purpose of the present section is to provide a \emph{feasibility guarantee}
for the certificate: under standard smoothness assumptions and mild geometric regularity
of the mirror map, there always exists a nontrivial range of sufficiently small stepsizes
for which the trajectory-wise generalized relative-smoothness condition holds.
This justifies certificate-driven line search (it must terminate) and explains the role
of the oracle mismatch through an explicit scaling law.

To discuss such feasibility analytically, we assume that the mirror map $\Phi$ is
$\sigma$-strongly convex and $\beta$-smooth with respect to the primal norm $\|\cdot\|$
for some constants $0<\sigma\le \beta<\infty$.
We also quantify the oracle mismatch by the relative error ratio

\begin{equation}
\Label{eq:delta_def}
\delta(\mathbf{y}):=
\begin{cases}
\dfrac{\|\Omega(\mathbf{y})-\nabla f(\mathbf{y})\|_*}
{\|\Omega(\mathbf{y})\|_*},&\Omega(\mathbf{y})\ne0,\\[1ex]
0,&\Omega(\mathbf{y})=\nabla f(\mathbf{y})=0,\\
+\infty,&\Omega(\mathbf{y})=0\ne\nabla f(\mathbf{y}).
\end{cases}
\end{equation}

As shown in Section \ref{S-thm:refined_eta_general_phi}, we have the following result.
\begin{theorem}[A sufficient stepsize condition for the trajectory-wise generalized relative-smoothness condition]
\Label{thm:refined_eta_general_phi}
Assume that $f$ is $L$-smooth with respect to the norm $\| \cdot \|$.
When $\eta$ satisfies
\begin{equation}
\Label{eq:eta_condition}
\eta \le \frac{\sigma^2}{L\beta}\left(1-\frac{2\beta}{\sigma}\delta(\mathbf{y})\right),
\end{equation}
then the trajectory-wise generalized relative-smoothness condition holds, i.e.,
\begin{equation}
\Label{eq:RS_general}
\eta\, D_{f,\Omega}(\mathbf{x}\|\mathbf{y}) \le D_\Phi(\mathbf{x}\|\mathbf{y}),
\qquad \text{for }x=F_{\eta\Omega}(\mathbf{y}).
\end{equation}
\end{theorem}

Theorem~\ref{thm:refined_eta_general_phi} should be viewed as a \emph{feasibility lemma}
for the certificate rather than a prescription that must be evaluated tightly in each application
when $\|\Omega(\mathbf{y})\|_*$ is not so small, i.e., the point 
$\mathbf{y}$ does not belong to the neighborhood of $\mathbf{x}_*$.
It ensures that, whenever $\delta(\mathbf{y})<\sigma/(2\beta)$, there exists a positive stepsize
range (in particular, sufficiently small $\eta$) 
for which \eqref{eq:RS_general} holds.
Consequently, a certificate-driven stepsize selection that checks \eqref{RS2}
(e.g., by decreasing $\eta$) is well-posed and guaranteed to find an admissible stepsize.
Moreover, \eqref{eq:eta_condition} makes explicit how the admissible scale of $\eta$
shrinks with the smoothness constant $L$, the mirror-map conditioning $\beta/\sigma$,
and the oracle mismatch $\delta(\mathbf{y})$.
When combined with the star-convexity interface condition, this feasibility guarantee
enables the certified last-iterate bound \eqref{DJ1}.

\section{Deterministic Finite-Difference Mirror Descent}\Label{S4}
This section instantiates the general $\Omega$-driven mirror-descent template with a
deterministic central finite-difference vector field.
Our goal is to apply the certified last-iterate bound of Theorem~\ref{Thm2} to this
zeroth-order construction.
To do so, we need two ingredients: (i) the trajectory-wise generalized relative-smoothness condition \eqref{RS2}
for the chosen stepsizes, and (ii) a generalized star-convexity interface outside
an exceptional region, of the form $\langle \Omega(\mathbf{x}),\mathbf{x}-\mathbf{x}_*\rangle \ge f(\mathbf{x})-f(\mathbf{x}_*)$ outside the exceptional region.
The key technical step is to manufacture (ii) for the finite-difference vector field:
we reduce it to a robust conic-dominance design problem (solved by Theorem~\ref{thm:alpha})
combined with a standard angle lower bound between $\mathbf{x}-\mathbf{x}_*$ and $\nabla f(\mathbf{x})$ under
Hessian bounds (Theorem~\ref{Thm5}).

\subsection{Vector field for deterministic finite-difference mirror descent}

We first define the deterministic finite-difference vector field used in the
mirror update.  Let $\mathbf{e}_i$ be the $i$th standard basis vector, and fix
$\epsilon>0$.  For $i=1,\ldots,d$, define
\(
m_i(\mathbf{x})
:=\frac{f(\mathbf{x}+\epsilon\mathbf{e}_i)
-f(\mathbf{x}-\epsilon\mathbf{e}_i)}{2\epsilon}\) and 
\(
r_i(\mathbf{x})
:=\frac{f(\mathbf{x}+\epsilon\mathbf{e}_i)
+f(\mathbf{x}-\epsilon\mathbf{e}_i)-2f(\mathbf{x})}{2\epsilon}.
\)
Let
\(
M(\mathbf{x}):=\|\mathbf{m}(\mathbf{x})\|_2,
R(\mathbf{x}):=\|\mathbf{r}(\mathbf{x})\|_2.
\)
For fixed $c\in(0,1]$, define the exceptional set
\begin{align}
V_{\epsilon,c}
&:=\bigl\{\mathbf{x}:M(\mathbf{x})=0\ \text{or}\
R(\mathbf{x})\ge cM(\mathbf{x})\bigr\}
=\bigl\{\mathbf{x}:M(\mathbf{x})\le R(\mathbf{x})/c\bigr\}.
\Label{DNB}
\end{align}
For $\mathbf{x}\notin V_{\epsilon,c}$, define
$\Omega_{\epsilon,c}^i(\mathbf{x}):=\alpha(\mathbf{x})m_i(\mathbf{x})$, where
\(
\alpha(\mathbf{x})
:=\frac{c\bigl(M(\mathbf{x})^2-R(\mathbf{x})^2\bigr)}
{M(\mathbf{x})\bigl(cM(\mathbf{x})-R(\mathbf{x})\bigr)}.
\)
The denominator is positive on the domain
$\mathbb{R}^d\setminus V_{\epsilon,c}$.

The procedure uses the following stopping rule.  If
$\mathbf{x}_j\in V_{\epsilon,c}$, it terminates at $\mathbf{x}_j$.  Otherwise,
$\Omega_{\epsilon,c}(\mathbf{x}_j)$ is well defined and the next mirror update
is performed.  Theorems~\ref{Thm1} and \ref{Thm2} can therefore be applied to
the performed updates once their respective assumptions have been verified.

The exceptional region contains the minimizer. Indeed, for each $i=1,\ldots,d$, the minimality of $\mathbf{x}_*$ and the convexity of the univariate function $t\mapsto f(\mathbf{x}_*+t\mathbf{e}_i)$ imply 
\( \frac{f(\mathbf{x}_*)-f(\mathbf{x}_*-\epsilon\mathbf{e}_i)} {\epsilon} \le 0 \le \frac{f(\mathbf{x}_*+\epsilon\mathbf{e}_i)-f(\mathbf{x}_*)} {\epsilon}. \) 
Since these two quantities equal $m_i(\mathbf{x}_*)-r_i(\mathbf{x}_*)$ and $m_i(\mathbf{x}_*)+r_i(\mathbf{x}_*)$, respectively, we have \( |m_i(\mathbf{x}_*)|\le r_i(\mathbf{x}_*). \) 
Consequently, \( M(\mathbf{x}_*)\le R(\mathbf{x}_*)\le \frac{R(\mathbf{x}_*)}{c}, \) where $c\in(0,1]$. Hence, $\mathbf{x}_*\in V_{\epsilon,c}$ by \eqref{DNB}.

\subsection{Generalized star-convexity outside the exceptional region}

We now show that the vector field defined above satisfies the generalized
star-convexity inequality required by Theorem~\ref{Thm2} outside
$V_{\epsilon,c}$.

Let $c\in(0,1]$ be a fixed constant. For a given (unknown)
$\boldsymbol{\xi}\neq0$, define the circular cone by the homogeneous inequality

\[
\mathcal{C}(\boldsymbol{\xi},c)
:=\left\{\mathbf{y}\in\mathbb{R}^d:
\langle\boldsymbol{\xi},\mathbf{y}\rangle
\ge c\|\boldsymbol{\xi}\|_2\|\mathbf{y}\|_2\right\}.
\]
For $\mathbf{y}\ne0$, this is equivalently the set of directions whose cosine
with $\boldsymbol{\xi}$ is at least $c$; the zero vector is included by the
homogeneous inequality.

The following theorem is proved in the appendix.

\begin{theorem}[Sufficient box scaling obtained from the containing ball]\Label{thm:alpha}
Given $\mathbf{l},\mathbf{h}\in\mathbb{R}^d$, define
$\mathbf{m}:=(\mathbf{l}+\mathbf{h})/2$ and
$\mathbf{r}:=(\mathbf{h}-\mathbf{l})/2$. Let $c\in(0,1]$ and suppose
$\|\mathbf{r}\|_2<c\|\mathbf{m}\|_2$.
Set
$\alpha:=\frac{c\bigl(\|\mathbf{m}\|_2^2-\|\mathbf{r}\|_2^2\bigr)}
{\|\mathbf{m}\|_2\bigl(c\|\mathbf{m}\|_2-\|\mathbf{r}\|_2\bigr)}$.
Then every $\boldsymbol{\xi}\in[\mathbf{l},\mathbf{h}]$ satisfies
\begin{align}
\langle\alpha\mathbf{m},\mathbf{y}\rangle
\ge\langle\boldsymbol{\xi},\mathbf{y}\rangle
\quad\text{for all }\mathbf{y}\in\mathcal{C}(\boldsymbol{\xi},c).
\Label{eq:alpha2}
\end{align}
For $d\ge2$, the displayed value is the exact minimum for the containing ball
$\mathbb{B}(\mathbf{m},\|\mathbf{r}\|_2)$.  For $d=1$, it remains a valid
sufficient scaling.  The finite-difference construction uses this single
sufficient formula in every dimension; only sufficiency is needed below.
\end{theorem}

The dominance requirement $\langle \mathbf{x}',\mathbf{y}\rangle\ge \langle \mathbf{x},\mathbf{y}\rangle$ for all $\mathbf{y}\in\mathcal C(\mathbf{x},c)$
is equivalent to the dual-cone constraint $\mathbf{x}'-\mathbf{x}\in \mathcal C(\mathbf{x},c)^{*}$ (a rotated second-order cone duality fact;
see, e.g., \cite{LoboEtAl1998,AlizadehGoldfarb2003,WilliamsonLec26}). 
Thus Theorem~\ref{thm:alpha} gives a sufficient solution of a \emph{robust conic dominance} design problem:
choose a single $\mathbf{x}'=\alpha \mathbf{m}$ that satisfies this dual-cone dominance uniformly over the admissible $\mathbf{x}$,
and our proof provides an explicit sufficient scaling $\alpha$.

In order that Theorem \ref{thm:alpha} guarantees that 
the above vector field satisfies 
generalized star-convexity outside the exceptional region
$V_{\epsilon,c}$,
we need to find a lower bound of 
$\frac{\langle \nabla f(\mathbf{x}),\,\mathbf{x}-\mathbf{x}_*\rangle}{\|\mathbf{x}-\mathbf{x}_*\|_2\,\|\nabla f(\mathbf{x})\|_2}$.
For this aim, we employ several concepts, strong convexity, smoothness,
and uniform Hessian bounds.

Now, we employ the framework given in Section \ref{S5:toolbox} when the norm is given as
the Euclid norm $\|~\|_2$.
In this case, for a twice differentiable function $f\in C^2$, 
$\mu$-strongly convexity is equivalent to
$\mu I \preceq \nabla^2 f(\mathbf{x}) $
for any vector $\mathbf{x}\in\mathbb{R}^d$,
see, e.g., \cite[Sec.~3.2]{EE227C_lecture03}.
Also, $L$-smoothness is equivalent to the condition 
$\nabla^2 f(\mathbf{x}) \preceq L I$
for any vector $x\in\mathbb{R}^d$ \cite[Proposition 3.1]{FawziLecture3}.
In summary, a convex function $f\in C^2$ 
is $\mu$-strongly convex and $L$-smooth
if and only if the convex function $f$ has the uniform Hessian bounds, i.e., 
\begin{equation}\Label{eq:hess_bounds}
\mu I \preceq \nabla^2 f(\mathbf{x}) \preceq L I
\qquad \text{for all }\mathbf{x}\in\mathbb{R}^d.
\end{equation}

As shown in Section \ref{S-Thm5}, we have the following theorem.
\begin{theorem}\Label{Thm5}
Assume that 
a convex function $f\in C^2$ satisfies the condition 
\eqref{eq:hess_bounds}.
Then, for every $\mathbf{x}\ne\mathbf{x}_*$, we have
\begin{align}
\frac{\langle \nabla f(\mathbf{x}),\,\mathbf{x}-\mathbf{x}_*\rangle}{\|\mathbf{x}-\mathbf{x}_*\|_2\,\|\nabla f(\mathbf{x})\|_2}
\ge \frac{2\sqrt{\mu L}}{\mu+L}
\Label{DJJ}.
\end{align}
\end{theorem}

We also record the elementary uncertainty-box inclusion needed below.  Convexity
of each univariate restriction $t\mapsto f(\mathbf{x}+t\mathbf{e}_i)$ gives
\begin{equation}
\frac{f(\mathbf{x})-f(\mathbf{x}-\epsilon\mathbf{e}_i)}{\epsilon}
\le \partial_i f(\mathbf{x})\le
\frac{f(\mathbf{x}+\epsilon\mathbf{e}_i)-f(\mathbf{x})}{\epsilon}.
\end{equation}
The two endpoints are respectively $m_i(\mathbf{x})-r_i(\mathbf{x})$ and
$m_i(\mathbf{x})+r_i(\mathbf{x})$.  Hence, we have
\begin{equation}
\nabla f(\mathbf{x})\in
[\mathbf{m}(\mathbf{x})-\mathbf{r}(\mathbf{x}),
 \mathbf{m}(\mathbf{x})+\mathbf{r}(\mathbf{x})].
\Label{eq:gradient-box}
\end{equation}

The combination of Theorem \ref{thm:alpha}, Theorem \ref{Thm5}, and
\eqref{eq:gradient-box} yields the following corollary.

\begin{corollary}\Label{LL9}
Assume that a convex function $f\in C^2$ satisfies the condition 
\eqref{eq:hess_bounds}, and we set
\begin{align}
c:=\frac{2\sqrt{\mu L}}{\mu+L}.\Label{NHJK}
\end{align}
When $\mathbf{x}$ does not belong to $V_{\epsilon,c}$, 
we have $D_{f,\Omega_{\epsilon,c}}(\mathbf{x}_*\|\mathbf{x}) \ge 0$, i.e., 
\begin{align}
\langle \Omega_{\epsilon,c} (\mathbf{x}), \mathbf{x} - \mathbf{x}_* \rangle
\ge f(\mathbf{x}) - f(\mathbf{x}_*) .
\end{align}
\end{corollary}

\begin{proof}
Since $f$ is convex, we have
$\langle \nabla f (\mathbf{x}), \mathbf{x} - \mathbf{x}_* \rangle
\ge f(\mathbf{x}) - f(\mathbf{x}_*) $.
It is sufficient to show
\begin{align}
\langle \Omega_{\epsilon,c} (\mathbf{x}), \mathbf{x} - \mathbf{x}_* \rangle
\ge
\langle \nabla f (\mathbf{x}), \mathbf{x} - \mathbf{x}_* \rangle.
\Label{BN2}
\end{align}

Theorem \ref{Thm5} guarantees the relation \eqref{DJJ}, i.e.,
\begin{align}
\mathbf{x}-\mathbf{x}_* \in \mathcal{C}( \nabla f(\mathbf{x}),c).
\Label{BN1N}
\end{align}
Then, we apply Theorem \ref{thm:alpha} to the case when 
$\mathbf{h}= \mathbf{m}(\mathbf{x})+\mathbf{r}(\mathbf{x})$, 
$\mathbf{l} = \mathbf{m}(\mathbf{x})-\mathbf{r}(\mathbf{x})$, 
$\mathbf{m}=\mathbf{m}(\mathbf{x})$,
and $\mathbf{x}= \nabla f (\mathbf{x})$.
Hence, the combination of \eqref{eq:alpha2} and \eqref{BN1N} implies \eqref{BN2}.
\qed\end{proof}

We have thus verified the generalized star-convexity condition required by
Theorem~\ref{Thm2} outside the exceptional region $V_{\epsilon,c}$
(equivalently, outside the set defined by \eqref{DNB}).
Combining this interface property with the trajectory-wise generalized relative-smoothness condition \eqref{RS2}
(for the same vector field $\Omega_{\epsilon,c}$), we can now invoke
Theorem~\ref{Thm2} to obtain an explicit last-iterate function-value bound for the
deterministic finite-difference mirror descent scheme.

\subsection{Certified convergence guarantee for deterministic finite differences}
We now state the resulting certified convergence bound obtained by combining
Theorem~\ref{Thm2} with Corollary~\ref{LL9} for the finite-difference vector field
$\Omega_{\epsilon,c}$.

\begin{theorem}\Label{Thm3}
Assume that a convex function $f\in C^2(\mathbb{R}^d)$ satisfies
\eqref{eq:hess_bounds}, and set $c$ as in \eqref{NHJK}.  Run the
finite-difference procedure with the stopping convention stated after
\eqref{DNB}, and suppose that every performed step satisfies \eqref{RS2}.
Then either the procedure first stops at an index $\tau\le t$, in which case
\begin{equation}
f(\mathbf{x}_\tau)-f(\mathbf{x}_*)
\le \sup_{\mathbf{x}\in V_{\epsilon,c}}
\bigl(f(\mathbf{x})-f(\mathbf{x}_*)\bigr),
\end{equation}
or no stopping occurs through $\mathbf{x}_t$, in which case
\begin{align}
f(\mathbf{x}_t)-f(\mathbf{x}_*)
\le
\frac{D_\Phi(\mathbf{x}_*\|\mathbf{x}_1)}
{\sum_{j=1}^{t-1}\eta_j}.
\Label{DJ3}
\end{align}
Equivalently, for the returned point $\widehat{\mathbf{x}}_t$, equal to the first
stopping point if one occurs and to $\mathbf{x}_t$ otherwise, we have
\begin{align}
f(\widehat{\mathbf{x}}_t)-f(\mathbf{x}_*)
\le \max\left\{
\frac{D_\Phi(\mathbf{x}_*\|\mathbf{x}_1)}
{\sum_{j=1}^{t-1}\eta_j},
\sup_{\mathbf{x}\in V_{\epsilon,c}}
\bigl(f(\mathbf{x})-f(\mathbf{x}_*)\bigr)\right\},
\end{align}
where the first term is omitted if stopping occurs before it is needed.
\end{theorem}

The bound \eqref{DJ3} has a transparent ``rate-plus-floor'' structure: outside the exceptional set $V_{\epsilon,c}$, the certified term decreases according to $\bigl(\sum_{j=1}^{t-1}\eta_j\bigr)^{-1}$, while the finite-difference resolution manifests itself through the error-floor term $\sup_{\mathbf{x}\in V_{\epsilon,c}} \bigl(f(\mathbf{x})-f(\mathbf{x}_*)\bigr)$. 
In particular, the certified term is $O(1/t)$ if the accepted stepsizes are uniformly bounded below by a positive constant.

We define the neighborhood $U_{\epsilon}(\mathbf{x}_0):=
\{\mathbf{x}: \|\mathbf{x}-\mathbf{x}_0\|_2 \le \epsilon \}$.
As shown in Section \ref{S-Thm6}, we have the following theorem.

\begin{theorem}\Label{Thm6}
Assume that $f$ is $\mu$-strongly convex and $L$-smooth with respect to the Euclidean norm, where $0<\mu\le L$.
We set $c$ as \eqref{NHJK}.
Then, we have
\begin{align}
V_{\epsilon,c} \subset
U_{\frac{L}{2\mu}\left(1+\frac1c\right)\epsilon\sqrt d}(\mathbf{x}_*)
\Label{DJ9}.
\end{align}
\end{theorem}

Theorem~\ref{Thm6} shows that the finite-difference exceptional set $V_{\epsilon,c}$
is contained in an explicit ball around the minimizer, with radius proportional to
$\epsilon\sqrt{d}$ with radius factor $\frac{L}{2\mu}(1+1/c)$.
Consequently, the error-floor term in \eqref{DJ3} can be upper bounded by the objective
variation on this neighborhood, yielding a more interpretable ``resolution-limited''
guarantee stated in the corollary below.
\begin{corollary}\Label{Thm7}
Assume that a convex function $f\in C^2(\mathbb{R}^d)$ satisfies the condition 
\eqref{eq:hess_bounds}.
We set $c$ as \eqref{NHJK}.
Run the finite-difference procedure with the stopping convention stated after \eqref{DNB}, and suppose that every performed step satisfies the trajectory-wise generalized relative-smoothness condition \eqref{RS2} for $\Omega_{\epsilon,c}$. 
Then,
the algorithm based on $\Omega_{\epsilon,c}$ satisfies 
\begin{align}
 f(\widehat{\mathbf{x}}_t)- f(\mathbf{x}_*)
\le 
\max \Big(\frac{D_\Phi(\mathbf{x}_* \| \mathbf{x}_1)}{\sum_{j=1}^{t-1} \eta_j}, 
\max_{\mathbf{x}\in U_{\frac{L}{2\mu}\left(1+\frac1c\right)\epsilon\sqrt d}(\mathbf{x}_*) }
( f(\mathbf{x})- f(\mathbf{x}_*))\Big)
\Label{DJ3B}.
\end{align}
where the first term is omitted if stopping occurs before it is needed.
\end{corollary}

Combining Corollary~\ref{LL9} (the generalized star-convexity interface outside
the exceptional region for
$\Omega_{\epsilon,c}$) with the trajectory-wise generalized relative-smoothness condition \eqref{RS2},
we obtain the certified last-iterate bound \eqref{DJ3B} for deterministic finite
differences.
Thus, the general $\Omega$-driven mirror-descent theory specializes to a fully
deterministic zeroth-order method with an explicit error floor determined by the
finite-difference resolution.

\section{Proof of Theorems \ref{Thm1} and \ref{Thm2}}\Label{S6}
\subsection{Two key lemmas}
To show Theorems \ref{Thm1} and \ref{Thm2}, we prepare two key lemmas.
To state our key lemmas, we define a two-variable function as
\begin{align}
J_{f,\eta,\Omega}(\mathbf{x},\mathbf{y}) := 
\eta(  f(\mathbf{x})- D_{f,\Omega}(\mathbf{x}\|\mathbf{y}))
+D_\Phi(\mathbf{x}\|\mathbf{y}).
\end{align}

Without any assumption for $f$, 
we have the following two lemmas.

\begin{lemma}\Label{Lem1}
The following relation holds:
\begin{align}
\begin{aligned}
J_{f,\eta,\Omega}(\mathbf{x},\mathbf{y}) 
=& D_\Phi(\mathbf{x}\| F_{\eta \Omega}(\mathbf{y}))
- \Phi^*(\nabla \Phi(\mathbf{y})-\eta \Omega(\mathbf{y}) ) 
 + \Phi^*(\nabla \Phi(\mathbf{y})) \\
&-\eta(\langle \Omega(\mathbf{y}), \mathbf{y} \rangle
-f(\mathbf{y})) .
\Label{BN1}
\end{aligned}
\end{align}
\end{lemma}

Lemma~\ref{Lem1} plays the role of the Bregman three-point identity step used in the classical mirror-descent proof
(e.g., \cite{FawziLecture11Bregman2023}, between (10) and (11)), but is written in a form compatible with a general vector field $\Omega$.

\begin{proof}
We have
\begin{align}
\begin{aligned}
&J_{f,\eta,\Omega}(\mathbf{x},\mathbf{y})  \\
=& 
(\langle 
\nabla \Phi(\mathbf{x})-\nabla \Phi(\mathbf{y}), \mathbf{x} \rangle
-\Phi^*(\nabla \Phi(\mathbf{x})) + \Phi^*(\nabla \Phi(\mathbf{y})) )
\\
&-
\eta(\langle \Omega(\mathbf{y}), \mathbf{y} - \mathbf{x} \rangle
-f(\mathbf{y})) \\
=& 
\langle 
\nabla \Phi(\mathbf{x})-\nabla \Phi(\mathbf{y})+\eta \Omega(\mathbf{y}), \mathbf{x} \rangle
\\
&-\Phi^*(\nabla \Phi(\mathbf{x})) + \Phi^*(\nabla \Phi(\mathbf{y})) 
-
\eta(\langle \Omega(\mathbf{y}), \mathbf{y} \rangle
-f(\mathbf{y})) \\
=& 
\langle 
\nabla \Phi(\mathbf{x})-\nabla \Phi(\mathbf{y})+\eta \Omega(\mathbf{y}), \mathbf{x} \rangle
\\
&-\Phi^*(\nabla \Phi(\mathbf{x}))
 + \Phi^*(\nabla \Phi(\mathbf{y})-\eta \Omega(\mathbf{y}) )\\
&- \Phi^*(\nabla \Phi(\mathbf{y})-\eta \Omega(\mathbf{y}) ) 
 + \Phi^*(\nabla \Phi(\mathbf{y})) 
-\eta(\langle \Omega(\mathbf{y}), \mathbf{y} \rangle
-f(\mathbf{y})) \\
=& D_\Phi(\mathbf{x}\| \nabla \Phi^{-1}(\nabla \Phi(\mathbf{y})-
\eta \Omega(\mathbf{y})))
\\
&- \Phi^*(\nabla \Phi(\mathbf{y})-\eta \Omega(\mathbf{y}) ) 
 + \Phi^*(\nabla \Phi(\mathbf{y})) 
-\eta(\langle \Omega(\mathbf{y}), \mathbf{y} \rangle
-f(\mathbf{y})) \\
=& D_\Phi(\mathbf{x}\| F_{\eta \Omega}(\mathbf{y}))
- \Phi^*(\nabla \Phi(\mathbf{y})-\eta \Omega(\mathbf{y}) ) 
 + \Phi^*(\nabla \Phi(\mathbf{y})) 
-\eta(\langle \Omega(\mathbf{y}), \mathbf{y} \rangle
-f(\mathbf{y})) .
\end{aligned}
\end{align}
\qed\end{proof}

\begin{lemma}\Label{Lem2}
\begin{align}
\begin{aligned}
&D_\Phi(\mathbf{x} \| \mathbf{y})-D_\Phi(\mathbf{x} \| F_{\eta \Omega}(\mathbf{y})) \notag\\
=&\eta(f(F_{\eta \Omega}(\mathbf{y}))-f(\mathbf{x}) ) 
+D_\Phi(F_{\eta \Omega}(\mathbf{y})\|\mathbf{y})-\eta D_{f,\Omega}(F_{\eta \Omega}(\mathbf{y})\|\mathbf{y})
+\eta D_{f,\Omega}(\mathbf{x}\|\mathbf{y}).
\end{aligned}
\end{align}
\end{lemma}

\begin{proof}
Lemma \ref{Lem1} implies that
\begin{align*}
\begin{aligned}
& D_\Phi(\mathbf{x}\|\mathbf{y})
- \eta(D_{f,\Omega}(\mathbf{x}\|\mathbf{y})- f(\mathbf{x})) \\
=& D_\Phi(\mathbf{x}\| F_{\eta \Omega}(\mathbf{y}))
- \Phi^*(\nabla \Phi(\mathbf{y})-\eta \Omega(\mathbf{y}) ) 
 + \Phi^*(\nabla \Phi(\mathbf{y})) 
-\eta(\langle \Omega(\mathbf{y}), \mathbf{y} \rangle
-f(\mathbf{y})) .
\end{aligned}
\end{align*}
Using this relation, we have
\begin{align*}
\begin{aligned}
&D_\Phi(\mathbf{x} \| \mathbf{y})-D_\Phi(\mathbf{x} \| F_{\eta \Omega}(\mathbf{y})) \\
=
&- \Phi^*(\nabla \Phi(\mathbf{y})-\eta \Omega(\mathbf{y}) ) 
 + \Phi^*(\nabla \Phi(\mathbf{y})) 
-\eta(\langle \Omega(\mathbf{y}), \mathbf{y} \rangle
-f(\mathbf{y})) \\
&+\eta(D_{f,\Omega}(\mathbf{x}\|\mathbf{y})- f(\mathbf{x}))\\
=&
\langle 
\nabla \Phi(F_{\eta \Omega}(\mathbf{y}))-\nabla \Phi(\mathbf{y}), F_{\eta \Omega}(\mathbf{y}) \rangle 
- \Phi^*(\nabla \Phi(F_{\eta \Omega}(\mathbf{y})) ) 
 + \Phi^*(\nabla \Phi(\mathbf{y})) \\
&-\langle 
\nabla \Phi(F_{\eta \Omega}(\mathbf{y}))-\nabla \Phi(\mathbf{y}), F_{\eta \Omega}(\mathbf{y}) \rangle 
-\eta(\langle \Omega(\mathbf{y}), \mathbf{y} \rangle
-f(\mathbf{y}) +f(\mathbf{x}))\\
& +\eta D_{f,\Omega}(\mathbf{x}\|\mathbf{y})\\
=&D_\Phi(F_{\eta \Omega}(\mathbf{y})\|\mathbf{y})
-\langle 
\nabla \Phi(F_{\eta \Omega}(\mathbf{y}))-\nabla \Phi(\mathbf{y}), F_{\eta \Omega}(\mathbf{y}) \rangle \\
&-\eta(\langle \Omega(\mathbf{y}), \mathbf{y} \rangle
-f(\mathbf{y}) +f(\mathbf{x})) 
+\eta D_{f,\Omega}(\mathbf{x}\|\mathbf{y})\\
=&D_\Phi(F_{\eta \Omega}(\mathbf{y})\|\mathbf{y})\\
&-
\eta\big(\langle \Omega(\mathbf{y}), \mathbf{y} -F_{\eta \Omega}(\mathbf{y})\rangle
-f(\mathbf{y}) +f(F_{\eta \Omega}(\mathbf{y}))
-f(F_{\eta \Omega}(\mathbf{y}))
+f(\mathbf{x})\big) \\
&+\eta D_{f,\Omega}(\mathbf{x}\|\mathbf{y})\\
=&
\eta(f(F_{\eta \Omega}(\mathbf{y}))-f(\mathbf{x}) ) 
+D_\Phi(F_{\eta \Omega}(\mathbf{y})\|\mathbf{y})-\eta D_{f,\Omega}(F_{\eta \Omega}(\mathbf{y})\|\mathbf{y})
+\eta D_{f,\Omega}(\mathbf{x}\|\mathbf{y}).
\end{aligned}
\end{align*}
\qed\end{proof}

\subsection{Proof of Theorem \ref{Thm1}}
It is sufficient to show 
\begin{align}
f(F_{\eta \Omega}(\mathbf{x}))
\le f(\mathbf{x}).\Label{BN}
\end{align}
The global generalized relative-smoothness condition \eqref{RS} implies
\begin{align}
J_{f,\eta,\Omega}(\mathbf{x},\mathbf{y})
= \eta f(\mathbf{x}) +
D_\Phi(\mathbf{x}\|\mathbf{y})- \eta D_{f,\Omega}(\mathbf{x}\|\mathbf{y})
\ge \eta f(\mathbf{x})
= J_{f,\eta,\Omega}(\mathbf{x},\mathbf{x}).
\end{align}
Thus, we have
$\min_{\mathbf{y}} J_{f,\eta,\Omega}(\mathbf{x},\mathbf{y})
=J_{f,\eta,\Omega}(\mathbf{x},\mathbf{x})$.

Since only the first term 
$D_\Phi(\mathbf{x}\| F_{\eta \Omega}(\mathbf{y}))$
of RHS of \eqref{BN1} depends on $\mathbf{x}$
and the minimum 
$\min_{\mathbf{x}} D_\Phi(\mathbf{x}\| F_{\eta \Omega}(\mathbf{y}))$ is realized with
$\mathbf{x}=F_{\eta \Omega}(\mathbf{y})$, we have
$\min_{\mathbf{x}} J_{f,\eta,\Omega}(\mathbf{x},\mathbf{y})
=J_{f,\eta,\Omega}(F_{\eta \Omega}(\mathbf{y}),\mathbf{y})$.
Thus, we have
\begin{align*}
\eta f(\mathbf{x}) &=J_{f,\eta,\Omega}(\mathbf{x},\mathbf{x})
\ge J_{f,\eta,\Omega}(F_{\eta \Omega}(\mathbf{x}),\mathbf{x})
\ge J_{f,\eta,\Omega}(F_{\eta \Omega}(\mathbf{x}),F_{\eta \Omega}(\mathbf{x}))\\
&=\eta f(F_{\eta \Omega}(\mathbf{x})),
\end{align*}
which implies \eqref{BN}.

\subsection{Proof of Theorem \ref{Thm2}}
We first derive the monotonicity property needed in both cases of the
proof.  For every $j=1,\ldots,t-1$, apply Lemma~\ref{Lem2} with
$\mathbf{x}=\mathbf{y}=\mathbf{x}_j,
\eta=\eta_j,
F_{\eta_j\Omega}(\mathbf{x}_j)=\mathbf{x}_{j+1}$.
Since
$D_\Phi(\mathbf{x}_j\|\mathbf{x}_j)=0$
and
$D_{f,\Omega}(\mathbf{x}_j\|\mathbf{x}_j)=0$,
Lemma~\ref{Lem2} gives
\begin{align}
\eta_j
\bigl(f(\mathbf{x}_{j+1})-f(\mathbf{x}_j)\bigr)
=
-D_\Phi(\mathbf{x}_j\|\mathbf{x}_{j+1})
-
\Bigl(
D_\Phi(\mathbf{x}_{j+1}\|\mathbf{x}_j)
-\eta_jD_{f,\Omega}
(\mathbf{x}_{j+1}\|\mathbf{x}_j)
\Bigr).
\Label{eq:Thm2-one-step-identity}
\end{align}
By the trajectory-wise generalized relative-smoothness condition \eqref{RS2},
\begin{equation}
D_\Phi(\mathbf{x}_{j+1}\|\mathbf{x}_j)
-\eta_jD_{f,\Omega}
(\mathbf{x}_{j+1}\|\mathbf{x}_j)
\ge 0.
\Label{eq:Thm2-RS-nonnegative}
\end{equation}
Since the Bregman divergence
$D_\Phi(\mathbf{x}_j\|\mathbf{x}_{j+1})$ is also nonnegative,
\eqref{eq:Thm2-one-step-identity} and
\eqref{eq:Thm2-RS-nonnegative} imply
\begin{equation}
f(\mathbf{x}_{j+1})\le f(\mathbf{x}_j),
\qquad
j=1,\ldots,t-1.
\Label{eq:Thm2-monotonicity}
\end{equation}

We now distinguish two cases.
First, assume that there exists an index
$k\in\{1,\ldots,t\}$ such that $\mathbf{x}_k\in U$.
Since $\mathbf{x}_k\in U$, the definition of the supremum gives
\begin{equation}
f(\mathbf{x}_k)-f(\mathbf{x}_*)
\le
\sup_{\mathbf{x}\in U}
\bigl(f(\mathbf{x})-f(\mathbf{x}_*)\bigr).
\Label{eq:Thm2-gap-in-U}
\end{equation}
If $k<t$, repeated application of \eqref{eq:Thm2-monotonicity}
for $j=k,\ldots,t-1$ gives
\begin{equation}
f(\mathbf{x}_t)\le f(\mathbf{x}_k).
\Label{eq:Thm2-after-entering-U}
\end{equation}
When $k=t$, \eqref{eq:Thm2-after-entering-U} holds with equality.
Combining \eqref{eq:Thm2-gap-in-U} and
\eqref{eq:Thm2-after-entering-U}, we obtain
\begin{equation}
f(\mathbf{x}_t)-f(\mathbf{x}_*)
\le
\sup_{\mathbf{x}\in U}
\bigl(f(\mathbf{x})-f(\mathbf{x}_*)\bigr).
\Label{eq:Thm2-U-case}
\end{equation}
Hence, \eqref{DJ1} holds in this case.

Next, we assume that $\mathbf{x}_j \notin U$ for $j=1,2,\ldots,t$.
The trajectory-wise generalized relative-smoothness condition \eqref{RS2} and
the star-convexity of $f$ imply
\begin{align}
&D_\Phi(F_{\eta_j\Omega}(\mathbf{x}_j)\|\mathbf{x}_j)
-\eta_j D_{f,\Omega}(F_{\eta_j\Omega}(\mathbf{x}_j)\|\mathbf{x}_j)
+\eta_j D_{f,\Omega}(\mathbf{x}_*\|\mathbf{x}_j)
\ge 0
\Label{NM}
\end{align}
for $j=1,2,\ldots,t-1$.
Thus, applying Lemma~\ref{Lem2} with
$\mathbf{x}=\mathbf{x}_*$, $\mathbf{y}=\mathbf{x}_j$, and
$\eta=\eta_j$, and using
$\mathbf{x}_{j+1}=F_{\eta_j\Omega}(\mathbf{x}_j)$, we obtain
\begin{align}
D_\Phi(\mathbf{x}_*\|\mathbf{x}_j)
-D_\Phi(\mathbf{x}_*\|\mathbf{x}_{j+1})
&\stackrel{(a)}{\ge}
\eta_j\bigl(f(\mathbf{x}_{j+1})-f(\mathbf{x}_*)\bigr)
\notag\\
&\stackrel{(b)}{\ge}
\eta_j\bigl(f(\mathbf{x}_t)-f(\mathbf{x}_*)\bigr),
\end{align}
where $(a)$ follows from Lemma~\ref{Lem2} and \eqref{NM}, 
and
$(b)$ follows from \eqref{eq:Thm2-monotonicity}.
Taking the sum, we have
\begin{align*}
\begin{aligned}
&D_\Phi(\mathbf{x}_*\|\mathbf{x}_1)
\ge
D_\Phi(\mathbf{x}_*\|\mathbf{x}_1)
-D_\Phi(\mathbf{x}_*\|\mathbf{x}_t)
\\
=&
\sum_{j=1}^{t-1}
\left[
D_\Phi(\mathbf{x}_*\|\mathbf{x}_j)
-D_\Phi(\mathbf{x}_*\|\mathbf{x}_{j+1})
\right]
\ge
\left(\sum_{j=1}^{t-1}\eta_j\right)
\bigl(f(\mathbf{x}_t)-f(\mathbf{x}_*)\bigr).
\end{aligned}
\end{align*}
Since $\sum_{j=1}^{t-1}\eta_j>0$, this implies
\begin{align}
f(\mathbf{x}_t)-f(\mathbf{x}_*)
\le
\frac{D_\Phi(\mathbf{x}_*\|\mathbf{x}_1)}
{\sum_{j=1}^{t-1}\eta_j},
\end{align}
and hence \eqref{DJ1} holds in this case.

These two cases exhaust all
possibilities, and therefore prove Theorem~\ref{Thm2}.
\qed

\section{Proof of Theorem \ref{thm:refined_eta_general_phi}}\Label{S-thm:refined_eta_general_phi}
We proceed in three steps.

\noindent\textbf{Step 1: Relation among $\|\mathbf{x}-\mathbf{y}\|$, $\eta\|\Omega(\mathbf{y})\|_*$, and 
$D_\Phi(\mathbf{x}\|\mathbf{y})$.}
Let $\mathbf{u}:=\nabla\Phi(\mathbf{y})$ and $\mathbf{v}:=\eta\Omega(\mathbf{y})$.
Since $\mathbf{x}=F_{\eta\Omega}(\mathbf{y})$, we have
$\mathbf{x}=\nabla\Phi^*(\mathbf{u}-\mathbf{v})$ and $\mathbf{y}=\nabla\Phi^*(\mathbf{u})$,
where $\Phi^*$ is the convex conjugate of $\Phi$.

Since $\Phi$ is $\sigma$-strongly convex with respect to
$\|\cdot\|$, the conjugate-gradient bound
\eqref{eq:conjugate_grad_Lip} gives
\begin{align}
\|\mathbf{x}-\mathbf{y}\|=\|\nabla\Phi^*(\mathbf{u}-\mathbf{v})-\nabla\Phi^*(\mathbf{u})\|
\le \frac{1}{\sigma}\|\mathbf{v}\|_* = \frac{\eta}{\sigma}\|\Omega(\mathbf{y})\|_*.
\Label{MK4}
\end{align}
Since $\Phi$ is $\beta$-smooth with respect to
$\|\cdot\|$, the conjugate Bregman bound
\eqref{eq:conjugate_bregman_lower} gives
\begin{align}
D_{\Phi^*}(\mathbf{u}\|\mathbf{u}-\mathbf{v})\ge \frac{1}{2\beta}\|\mathbf{v}\|_*^2.
\end{align}
Using the Bregman duality identity
$D_\Phi(\mathbf{x}\|\mathbf{y})=D_{\Phi^*}(\nabla\Phi(\mathbf{y})\|\nabla\Phi(\mathbf{x}))$
and $\nabla\Phi(\mathbf{x})=\mathbf{u}-\mathbf{v}$, we get
\begin{equation}
\Label{eq:Dphi_lower_v}
D_\Phi(\mathbf{x}\|\mathbf{y})=D_{\Phi^*}(\mathbf{u}\|\mathbf{u}-\mathbf{v})\ge \frac{1}{2\beta}\|\mathbf{v}\|_*^2
= \frac{\eta^2}{2\beta}\|\Omega(\mathbf{y})\|_*^2.
\end{equation}

\noindent\textbf{Step 2: Bounding $D_{f,\Omega}$ by $\|\mathbf{x}-\mathbf{y}\|$ and the oracle mismatch.}
Decompose
\begin{align}
D_{f,\Omega}(\mathbf{x}\|\mathbf{y})
= \underbrace{\bigl(f(\mathbf{x})-f(\mathbf{y})-\langle \nabla f(\mathbf{y}),\mathbf{x}-\mathbf{y}\rangle\bigr)}_{=:D_f(\mathbf{x}\|\mathbf{y})}
+ \langle \nabla f(\mathbf{y})-\Omega(\mathbf{y}),\mathbf{x}-\mathbf{y}\rangle.
\Label{MK1}
\end{align}
The property \eqref{eq:Lsmooth_bregman_upper} of $L$-smoothness of $f$ implies
\begin{align}
D_f(\mathbf{x}\|\mathbf{y})\le \frac{L}{2}\|\mathbf{x}-\mathbf{y}\|^2.
\Label{MK2}
\end{align}
Cauchy--Schwarz inequality implies
\begin{align}
\langle \nabla f(\mathbf{y})-\Omega(\mathbf{y}),\mathbf{x}-\mathbf{y}\rangle
\le \|\nabla f(\mathbf{y})-\Omega(\mathbf{y})\|_*\,\|\mathbf{x}-\mathbf{y}\|.
\Label{MK3}
\end{align}
Substituting \eqref{MK2} and \eqref{MK3} into \eqref{MK1}, we have
\begin{equation}
\Label{eq:DfOmega_bound}
D_{f,\Omega}(\mathbf{x}\|\mathbf{y})\le \frac{L}{2}\|\mathbf{x}-\mathbf{y}\|^2 + \|\nabla f(\mathbf{y})-\Omega(\mathbf{y})\|_*\,\|\mathbf{x}-\mathbf{y}\|.
\end{equation}
Plugging the evaluation \eqref{MK4} into
\eqref{eq:DfOmega_bound}, we have
\begin{align}
D_{f,\Omega}(\mathbf{x}\|\mathbf{y})
\le \frac{L}{2}\frac{\eta^2}{\sigma^2}\|\Omega(\mathbf{y})\|_*^2
+ \frac{\eta}{\sigma}\|\nabla f(\mathbf{y})-\Omega(\mathbf{y})\|_*\,\|\Omega(\mathbf{y})\|_*.
\Label{MK7}
\end{align}

\noindent\textbf{Step 3: Enforcing $\eta D_{f,\Omega}\le D_\Phi$.}
Rearranging \eqref{eq:eta_condition}, we have
$\frac{L}{2}\frac{\eta}{\sigma^2} + \frac{1}{\sigma}\delta(\mathbf{y})\le \frac{1}{2\beta}$.
Multiplying $\eta^2>0$, we have
\begin{align}
\frac{L}{2}\frac{\eta^3}{\sigma^2} + \frac{\eta^2}{\sigma}\delta(\mathbf{y})
\le \frac{\eta^2}{2\beta}.
\Label{MK5}
\end{align}
Multiplying $\eta$ in \eqref{MK7} and using \eqref{eq:delta_def}, we have
\begin{align*}
\begin{aligned}
\eta D_{f,\Omega}(\mathbf{x}\|\mathbf{y})
\le  \left(\frac{L}{2}\frac{\eta^3}{\sigma^2}
+ \frac{\eta^2}{\sigma}\delta(\mathbf{y})\right)\|\Omega(\mathbf{y})\|_*^2 
\stackrel{(a)}{\le} & \frac{\eta^2}{2\beta}\|\Omega(\mathbf{y})\|_*^2
\stackrel{(b)}{\le} D_\Phi(\mathbf{x}\|\mathbf{y}),
\end{aligned}
\end{align*}
where $(a)$ follows from \eqref{MK5}, and
$(b)$ follows from \eqref{eq:Dphi_lower_v}.

\section{Proof of Theorem \ref{Thm6}}
\Label{S-Thm6}
By \eqref{DNB}, it is sufficient to show that every point 
$\mathbf{x}$ satisfying
\(
\|\mathbf{x}-\mathbf{x}_*\|_2
>
\frac{L}{2\mu}
\left(1+\frac{1}{c}\right)\epsilon\sqrt d
\)
also satisfies
$M(\mathbf{x})>R(\mathbf{x})/c$.
We establish this implication by bounding $R(\mathbf{x})$ from above
and $M(\mathbf{x})$ from below.

For each $i=1,\ldots,d$, the descent inequality applied to
$\mathbf{x}\pm\epsilon\mathbf{e}_i$, together with convexity, gives
\begin{align}
0
&\le
f(\mathbf{x}+\epsilon\mathbf{e}_i)
+f(\mathbf{x}-\epsilon\mathbf{e}_i)
-2f(\mathbf{x})
\le
L\epsilon^2.
\Label{eq:Thm6-second-difference-bound}
\end{align}
Hence,
\begin{equation}
0\le r_i(\mathbf{x})\le\frac{L\epsilon}{2},
\qquad
R(\mathbf{x})\le\frac{L\epsilon}{2}\sqrt d.
\Label{eq:Thm6-R-bound}
\end{equation}
Moreover, the fundamental theorem of calculus and the
$L$-Lipschitz continuity of $\nabla f$ give
\begin{align}
|m_i(\mathbf{x})-\partial_i f(\mathbf{x})|
\le
\frac{1}{2\epsilon}
\int_{-\epsilon}^{\epsilon}
\left|
\partial_i f(\mathbf{x}+s\mathbf{e}_i)
-\partial_i f(\mathbf{x})
\right|\,ds
\le
\frac{1}{2\epsilon}
\int_{-\epsilon}^{\epsilon}L|s|\,ds
=
\frac{L\epsilon}{2}.
\Label{eq:Thm6-coordinate-error}
\end{align}
Therefore,
\begin{equation}
M(\mathbf{x})
\ge
\|\nabla f(\mathbf{x})\|_2
-\frac{L\epsilon}{2}\sqrt d.
\Label{eq:Thm6-M-bound}
\end{equation}

Since $\mathbf{x}_*$ is a minimizer and $f$ is differentiable,
$\nabla f(\mathbf{x}_*)=0$.  Strong convexity and the
Cauchy--Schwarz inequality imply
\begin{equation}
\|\nabla f(\mathbf{x})\|_2
\ge
\mu\|\mathbf{x}-\mathbf{x}_*\|_2.
\Label{eq:Thm6-gradient-lower-bound}
\end{equation}
Consequently, if
\[
\|\mathbf{x}-\mathbf{x}_*\|_2
>
\frac{L}{2\mu}
\left(1+\frac{1}{c}\right)\epsilon\sqrt d,
\]
then \eqref{eq:Thm6-R-bound}--\eqref{eq:Thm6-gradient-lower-bound}
give
\begin{align}
M(\mathbf{x})-\frac{R(\mathbf{x})}{c}
&\ge
\mu\|\mathbf{x}-\mathbf{x}_*\|_2
-\frac{L\epsilon}{2}
\left(1+\frac{1}{c}\right)\sqrt d
>0.
\end{align}
Thus, $\mathbf{x}\notin V_{\epsilon,c}$ by \eqref{DNB}.
Taking the contraposition proves \eqref{DJ9}.
\qed

\section{Proof of Theorem \ref{Thm5}}\Label{S-Thm5}
We first prepare the matrix inequality used in the proof of
Theorem~\ref{Thm5}.

\begin{lemma}[Angle bound for SPD matrices]\Label{lem:spd-angle}
Let $A$ be SPD and assume $\mu I\preceq A\preceq LI$ with
$0<\mu\le L$.
Then for all $\mathbf{v}\neq 0$,
\begin{equation}\Label{eq:spd-angle}
\frac{\langle A\mathbf{v},\mathbf{v}\rangle}
{\|A\mathbf{v}\|_2\,\|\mathbf{v}\|_2}
\ge \frac{2\sqrt{\mu L}}{\mu+L}.
\end{equation}
\end{lemma}

\begin{proof}
We invoke a classical Kantorovich inequality.
For SPD $A$ with eigenvalues in $[\mu,L]$, any
$\mathbf{u}\neq 0$ satisfies the relation
\begin{equation}\Label{eq:kantorovich}
\langle \mathbf{u}, A \mathbf{u}\rangle
\langle \mathbf{u}, A^{-1}\mathbf{u}\rangle
\le
\frac{(\mu+L)^2}{4\mu L}
\langle\mathbf{u}, \mathbf{u}\rangle^2.
\end{equation}
(See, e.g., Newman (1960) \cite[p.33]{Newman1960}
for an elementary proof; Strang (1960) \cite[p.468]{Strang1960}
for a related generalization and sharpness discussion.)
Now set $\mathbf{u}:=A^{1/2}\mathbf{v}$ (well-defined since $A$ is SPD).
Then, we have the relations
$\langle\mathbf{u},\mathbf{u}\rangle
=\langle\mathbf{v},A\mathbf{v}\rangle
=\langle A\mathbf{v},\mathbf{v}\rangle$,
$\langle\mathbf{u},A\mathbf{u}\rangle
=\langle\mathbf{v},A^2\mathbf{v}\rangle
=\|A\mathbf{v}\|_2^2$,
and
$\langle\mathbf{u},A^{-1}\mathbf{u}\rangle
=\langle\mathbf{v},\mathbf{v}\rangle
=\|\mathbf{v}\|_2^2$.
Plugging these into \eqref{eq:kantorovich} yields
\(
\|A\mathbf{v}\|_2^2\,\|\mathbf{v}\|_2^2
\le
\frac{(\mu+L)^2}{4\mu L}
\langle A\mathbf{v},\mathbf{v}\rangle^2,
\)
which implies
\(
\left(
\frac{\langle A\mathbf{v},\mathbf{v}\rangle}
{\|A\mathbf{v}\|_2\,\|\mathbf{v}\|_2}
\right)^2
\ge
\frac{4\mu L}{(\mu+L)^2}.
\)
Taking square-roots (all quantities are nonnegative) proves
\eqref{eq:spd-angle}.
\qed
\end{proof}

We now prove Theorem~\ref{Thm5}.
Fix $\mathbf{x}\neq\mathbf{x}_*$ and set
$\mathbf{v}:=\mathbf{x}-\mathbf{x}_*\neq0$.
To prove \eqref{DJJ}, it is sufficient to construct a symmetric
positive definite matrix $\bar H$ satisfying
\begin{equation}
\nabla f(\mathbf{x})=\bar H\mathbf{v},
\qquad
\mu I\preceq\bar H\preceq LI,
\Label{eq:Thm5-proof-goal}
\end{equation}
because Lemma~\ref{lem:spd-angle}, applied to
$A=\bar H$, then gives the desired lower bound.

For completeness, define
\begin{equation}\Label{eq:cosdef}
\cos\theta(\mathbf{x})
:=
\frac{\langle \nabla f(\mathbf{x}),\,
\mathbf{x}-\mathbf{x}_*\rangle}
{\|\mathbf{x}-\mathbf{x}_*\|_2\,
\|\nabla f(\mathbf{x})\|_2}.
\end{equation}
Since $\mathbf{x}_*$ is a minimizer of the differentiable function
$f$, we have $\nabla f(\mathbf{x}_*)=0$.  By the fundamental theorem
of calculus,
\begin{equation}\Label{eq:FTC}
\nabla f(\mathbf{x})-\nabla f(\mathbf{x}_*)
=
\int_0^1
\frac{d}{dt}\nabla f(\mathbf{x}_*+t\mathbf{v})\,dt
=
\int_0^1
\nabla^2 f(\mathbf{x}_*+t\mathbf{v})\,
\mathbf{v}\,dt.
\end{equation}
Define the averaged Hessian
\begin{equation}\Label{eq:Hbar}
\bar H
:=
\int_0^1
\nabla^2 f(\mathbf{x}_*+t\mathbf{v})\,dt.
\end{equation}
It follows from \eqref{eq:FTC} and
$\nabla f(\mathbf{x}_*)=0$ that
\begin{equation}
\nabla f(\mathbf{x})=\bar H\mathbf{v}.
\Label{eq:Thm5-gradient-representation}
\end{equation}
Moreover, because the Loewner order is preserved under integration,
the uniform Hessian bounds \eqref{eq:hess_bounds} give
\begin{equation}\Label{eq:Hbar-bounds}
\mu I\preceq\bar H\preceq LI.
\end{equation}
Thus, \eqref{eq:Thm5-gradient-representation} and
\eqref{eq:Hbar-bounds} establish \eqref{eq:Thm5-proof-goal}.

Applying Lemma~\ref{lem:spd-angle} with
$A=\bar H$, and then using
\eqref{eq:Thm5-gradient-representation} and
$\mathbf{v}=\mathbf{x}-\mathbf{x}_*$, we obtain
\begin{align}
\frac{\langle\nabla f(\mathbf{x}),
\mathbf{x}-\mathbf{x}_*\rangle}
{\|\nabla f(\mathbf{x})\|_2
\|\mathbf{x}-\mathbf{x}_*\|_2}
=
\frac{\langle\bar H\mathbf{v},\mathbf{v}\rangle}
{\|\bar H\mathbf{v}\|_2\|\mathbf{v}\|_2}
\ge
\frac{2\sqrt{\mu L}}{\mu+L}.
\end{align}
This is precisely \eqref{DJJ}, and therefore proves
Theorem~\ref{Thm5}.
\qed

\section{Discussion and Conclusions}
\Label{sec:discussion}

\subsection{A unifying proof skeleton: from the Bregman three-point identity to $\Omega$-driven mirror descent}
A central technical message of this paper is that the classical mirror-descent convergence analysis,
ultimately powered by the Bregman three-point identity,
admits a structural generalization to the case where the gradient $\nabla f$ is replaced by a general vector field $\Omega$.
In the gradient case, the standard proof pattern uses the three-point identity to telescope the key inequality;
a clean instance appears in \cite{FawziLecture11Bregman2023}, where the step between (10) and (11) is exactly the
three-point identity insertion that converts an inner-product term into a difference of Bregman divergences.

Our Lemmas~\ref{Lem1}--\ref{Lem2} can be viewed as an $\Omega$-compatible reparameterization of this proof skeleton:
the auxiliary functional $J_{f,\eta,\Omega}$ is designed so that the mirror update interacts with Bregman geometry
in a way that preserves telescoping and yields explicit bounds once a suitable generalized relative-smoothness condition holds
(see Section~\ref{S6}). 
This viewpoint clarifies why several parts of our derivations may look familiar when $\Omega=\nabla f$:
in that case the argument follows the classical route based on the three-point identity.
However, the contribution is that the \emph{same} telescoping mechanism can be made to work for general $\Omega$,
thereby producing a single evaluation template that applies uniformly to disparate algorithms
once they can be cast as an $\Omega$-driven mirror update (Section~\ref{S2}).

\subsection{Trajectory-wise certification as a design principle}
A second message is methodological: instead of relying solely on global assumptions,
we emphasize a \emph{trajectory-wise} (a posteriori) condition (e.g., \eqref{RS2}) that can be verified along the realized iterates.
This shift is practically motivated (determinism and verifiability) and technically natural in our framework:
once the core proof identity is written in an $\Omega$-driven form,
the only place where problem regularity enters is through an inequality that compares $\eta D_{f,\Omega}$ to $D_\Phi$
along the actual update step (Theorem~\ref{Thm2}).
In this sense, the paper provides not just a convergence bound but also a \emph{certificate format}:
if the trajectory-wise inequality holds on a time window, then a last-iterate function-value guarantee follows automatically
(see \eqref{DJ1} in Theorem~\ref{Thm2}).

\subsection{Why generalized star-convexity outside an exceptional region becomes the relevant property}
One technically surprising aspect emerges when we instantiate the general theory with deterministic finite differences.
At first glance, natural analytic assumptions for finite differences are classical:
$f\in C^2(\mathbb{R}^d)$ is $\mu$-strongly convex and $L$-smooth, i.e., $f$ has uniform Hessian bounds (Section~\ref{S4}).
Indeed, such assumptions allow one to control approximation errors and to derive quantitative geometric bounds
(e.g., angle/inner-product lower bounds and neighborhood estimates; see Theorems~\ref{Thm5}--\ref{Thm7}).

However, once we pass through the general $\Omega$-driven mirror-descent template,
the property that \emph{interfaces} with the main evaluation inequality is not ``convexity'' per se,
but rather a star-convexity-type inequality written directly in terms of $\Omega$:
$\langle \Omega(\mathbf{x}),\, \mathbf{x}-\mathbf{x}_*\rangle \;\ge\; f(\mathbf{x})-f(\mathbf{x}_*)$,
and, for deterministic finite differences, its variant outside an exceptional region that allows an explicit error floor.
This interface condition is exactly the assumption used in Theorem~\ref{Thm2} and in the definition of
generalized star-convexity outside an exceptional region in Section~\ref{S2}.
In other words, the general evaluation bound reveals that the correct notion of ``convex-like geometry''
for last-iterate function-value guarantees is generalized star-convexity outside an exceptional region,
because this is exactly what turns the $\Omega$-driven progress term into a function-value gap (Theorem~\ref{Thm2}).

This perspective is somewhat counterintuitive: one may expect that convexity (or strong convexity) is the
primary structural assumption, yet in our $\Omega$-driven analysis convexity becomes a \emph{sufficient tool}
to establish the generalized star-convex inequality, while the latter is the \emph{actual interface condition}
required by the general template (Section~\ref{S2} and Section~\ref{S4}).
Thus, 
for deterministic zeroth-order mirror descent,
it is often more informative to ask whether an $\Omega$-dependent star-convexity inequality holds
(outside a small neighborhood) than to insist on classical convexity alone.

\subsection{Finite differences and the emergence of robust conic dominance}
A further ``hidden'' mechanism becomes visible only after adopting the general template.
To apply Theorem~\ref{Thm2} to deterministic finite-difference mirror descent, one must verify
the generalized star-convexity inequality outside the exceptional region for the specific vector field $\Omega_{\epsilon,c}$
constructed in Section~\ref{S4}.
This verification is not automatic: finite differences produce a surrogate direction whose inner products must be controlled
\emph{uniformly} against a cone of directions aligned with the true displacement/gradient.
Our analysis shows that this requirement is naturally phrased as a \emph{robust conic dominance} problem:
construct a single (scaled) proxy direction that dominates all admissible true directions on a circular cone,
uniformly over an uncertainty set induced by finite-difference errors.
This is precisely where conic duality, and in particular second-order-cone geometry, enters the analysis, although it is not apparent from the finite-difference formulas themselves. 

The result proved in the appendix determines how much the central-difference vector must be scaled so that its inner product dominates that of every gradient consistent with the observed function values, for all displacement directions in the prescribed cone. 
This uniform dominance verifies the required star-convexity interface for $\Omega_{\epsilon,c}$ and thereby allows the general $\Omega$-driven bound to be applied.

Conceptually, the role of robust conic dominance is to \emph{manufacture} the star-convexity interface inequality
for a zeroth-order surrogate $\Omega_{\epsilon,c}$.
Thus, the appearance of a conic-robustness subproblem is not an added complication,
but rather the mechanism that translates deterministic approximation error into a clean error-floor term
within the general last-iterate bound (Theorems~\ref{Thm3} and \ref{Thm7}).

\subsection{Limitations and future directions}

The certificate established here is finite-horizon and conditional on the
realized trajectory.  An asymptotic statement requires divergence of the
accumulated stepsizes, which is not guaranteed here for every possible
line-search rule.  Moreover, the conic formula is exact for the containing ball
when $d\ge2$ and is used only as a sufficient box scaling in the
finite-difference construction.

An interesting direction is to refine the robust conic dominance step for other deterministic direction sets
(orthogonal blocks, lattice stencils, deterministic quadrature designs) and to quantify how the corresponding uncertainty sets
change the size of the exceptional region (error floor).
Another direction is to explore whether the generalized star-convexity interface condition can be established
under weaker smoothness assumptions, potentially extending deterministic certificates to broader non-Euclidean regimes.

\begin{acknowledgements}
During the preparation of this manuscript, the author used Microsoft Copilot (enterprise/Exchange account, GPT5.2/GPT 5.6 Thinking) 
to assist in several stages of the research process, including performing initial literature surveys, refining grammatical expressions, and generating candidate formulations of Theorems 
~\ref{thm:refined_eta_general_phi}, \ref{thm:alpha}, \ref{Thm5}, \ref{Thm6} to explore possible proof routes. All statements, proofs, and references were critically assessed, corrected where necessary, and rigorously completed by the author, who takes full responsibility for the correctness and originality of the final manuscript.

This work was funded by the General R\&D Projects of 1+1+1 CUHK--CUHK(SZ)--GDST Joint Collaboration Fund (Grant No. GRDP2025-022), the Guangdong Provincial Quantum Science Strategic Initiative (Grant No. GDZX2505003), and the National Natural Science Foundation of China under Grant 62171212.
\end{acknowledgements}

\appendix

The purpose of this appendix is to analyze the uniform inner-product dominance condition used in the deterministic finite-difference construction of the main manuscript. Given an uncertainty set for the unknown gradient, the problem is to determine a single scaling of its center whose inner product dominates that of every admissible gradient over the corresponding family of directions. We refer to this geometric problem as robust conic dominance.

We first solve the robust conic-dominance problem exactly for a closed Euclidean ball. In dimensions at least two, we determine the exact common scaling; the one-dimensional case has a separate exact value. We then apply the ball result to the coordinate uncertainty box used in the main manuscript. Since the box is contained in a Euclidean ball with the same center and radius equal to the Euclidean norm of its half-width vector, the ball result yields a dimension-uniform sufficient scaling for the box.

This distinction is important. Exactness is established for the ball problem, whereas only sufficiency is transferred to the box by set inclusion. We first state the two target results, then prove the exact ball theorem through pointwise, radial, angular, and one-dimensional reductions, and finally deduce the box theorem from the containment relation.

\section{Robust conic dominance for balls and boxes}
\Label{S-target-results}

For $\mathbf{m}\in\mathbb{R}^d$ and $R\ge 0$, let
\[
\mathbb{B}(\mathbf{m},R)
:=
\{\mathbf{x}\in\mathbb{R}^d:
\|\mathbf{x}-\mathbf{m}\|_2\le R\}
\]
denote the closed Euclidean ball. For a nonzero vector $\mathbf{x}$ and $c\in(0,1]$, define
\[
\mathcal{C}(\mathbf{x},c)
:=
\{\mathbf{y}\in\mathbb{R}^d:
\langle\mathbf{x},\mathbf{y}\rangle
\ge c\|\mathbf{x}\|_2\|\mathbf{y}\|_2\}.
\]
Positive rescaling of the cone axis leaves this set unchanged: for every
nonzero $\mathbf{v}\in\mathbb{R}^d$ and every $\lambda>0$,
\begin{equation}
\mathcal{C}(\lambda\mathbf{v},c)
=
\mathcal{C}(\mathbf{v},c).
\Label{eq:cone-positive-scaling}
\end{equation}
Indeed, multiplication of the cone axis by $\lambda>0$ multiplies both
sides of the defining inequality by the same positive scalar.
For a cone $K\subseteq\mathbb{R}^d$, we use the dual-cone convention
\[
K^*
:=
\{\mathbf{w}\in\mathbb{R}^d:
\langle\mathbf{w},\mathbf{y}\rangle\ge 0
\text{ for every }\mathbf{y}\in K\}.
\]

Fix $\mathbf{m}\neq 0$, $c\in(0,1]$, and a nonzero point $\mathbf{x}$. We define the pointwise dominance value by
\begin{equation}
\alpha_{\mathrm{pt}}(\mathbf{x};\mathbf{m},c)
:=
\inf\Bigl\{
\alpha>0:
\langle\alpha\mathbf{m}-\mathbf{x},\mathbf{y}\rangle\ge 0
\text{ for every }\mathbf{y}\in\mathcal{C}(\mathbf{x},c)
\Bigr\}.
\Label{eq:pointwise-alpha}
\end{equation}
Thus $\mathbf{x}$ is fixed in \eqref{eq:pointwise-alpha}, whereas $\mathbf{y}$ ranges over the entire cone associated with that point.

For $0<R<c\|\mathbf{m}\|_2$, define the robust ball value by
\begin{align}
\alpha_{\mathrm{ball}}(\mathbf{m},R,c)
:=
\inf\Bigl\{
\alpha>0:
&\langle\alpha\mathbf{m}-\mathbf{x},\mathbf{y}\rangle\ge 0,
\notag\\[-1mm]
&\forall\mathbf{x}\in\mathbb{B}(\mathbf{m},R),
\quad
\forall\mathbf{y}\in\mathcal{C}(\mathbf{x},c)
\Bigr\}.
\Label{eq:ball-alpha}
\end{align}
The assumption $R<c\|\mathbf{m}\|_2\le\|\mathbf{m}\|_2$ implies that the ball does not contain the origin, so every cone occurring in \eqref{eq:ball-alpha} is well defined.

For vectors $\mathbf{l},\mathbf{h}\in\mathbb{R}^d$ satisfying
$l_i\le h_i$ for every $i$, define
\begin{align}
[\mathbf{l},\mathbf{h}]
:=
\{\boldsymbol{\xi}\in\mathbb{R}^d:
 l_i\le\xi_i\le h_i
\text{ for }i=1,\ldots,d\}.\label{BNA}
\end{align}

Our first target is the exact robust dominance theorem for the containing ball.

\begin{theorem}[Exact robust dominance scaling on a ball]
\Label{thm:alpha-B}
Let $\mathbf{m}\in\mathbb{R}^d\setminus\{0\}$,
$c\in(0,1]$, and $0<R<c\|\mathbf{m}\|_2$.
If $d\ge2$, then
\begin{align}
&\min\Bigl\{\alpha>0:
\langle\alpha\mathbf{m},\mathbf{y}\rangle
\ge\langle\mathbf{x},\mathbf{y}\rangle,
\quad
\forall\mathbf{x}\in\mathbb{B}(\mathbf{m},R),
\quad
\forall\mathbf{y}\in\mathcal{C}(\mathbf{x},c)
\Bigr\}
\notag\\
&\hspace{25mm}
=
\frac{c\bigl(\|\mathbf{m}\|_2^2-R^2\bigr)}
{\|\mathbf{m}\|_2\bigl(c\|\mathbf{m}\|_2-R\bigr)}.
\Label{MA}
\end{align}
If $d=1$, the exact minimum is
$1+R/\|\mathbf{m}\|_2$, and the right-hand side of \eqref{MA} remains a valid upper bound.
\end{theorem}

Our second target is the sufficient box theorem used in the main manuscript. Its proof will use the first theorem together with the box-containment lemma established at the end of this appendix.

\begin{theorem}[Sufficient scaling for a coordinate box]
\Label{thm:alpha}
Let $\mathbf{l},\mathbf{h}\in\mathbb{R}^d$ satisfy
$l_i\le h_i$ for every $i$, and define
$\mathbf{m}:=(\mathbf{l}+\mathbf{h})/2$ and
$\mathbf{r}:=(\mathbf{h}-\mathbf{l})/2$.
Let $c\in(0,1]$ and suppose that
$\|\mathbf{r}\|_2<c\|\mathbf{m}\|_2$.
Set
\[
\alpha
:=
\frac{c\bigl(\|\mathbf{m}\|_2^2-\|\mathbf{r}\|_2^2\bigr)}
{\|\mathbf{m}\|_2
\bigl(c\|\mathbf{m}\|_2-\|\mathbf{r}\|_2\bigr)}.
\]
Then every $\boldsymbol{\xi}\in[\mathbf{l},\mathbf{h}]$ satisfies
\begin{equation}
\langle\alpha\mathbf{m},\mathbf{y}\rangle
\ge
\langle\boldsymbol{\xi},\mathbf{y}\rangle
\qquad
\text{for every }
\mathbf{y}\in\mathcal{C}(\boldsymbol{\xi},c).
\Label{eq:alpha2}
\end{equation}
For $d\ge2$, the displayed value is the exact minimum for the containing ball
$\mathbb{B}(\mathbf{m},\|\mathbf{r}\|_2)$.
For $d=1$, the displayed value is feasible for the containing ball and
therefore sufficient for the box.  The finite-difference construction uses
this single sufficient formula in every dimension; only sufficiency is needed there.
\end{theorem}

\section{Proof of the exact ball theorem}
\Label{S-proof-ball}
\subsection{Normalization of the ball problem} \Label{S-normalization}
We prove Theorem~\ref{thm:alpha-B} by separating five logically distinct
parts of the argument.  We first normalize the scale.  We then determine, for each fixed unit direction $\mathbf{u}$, all positive
radial coordinates $z$ for which
$z\mathbf{u}\in\mathbb{B}(\mathbf{m},\rho)$, and solve the dominance
problem at one fixed point, including the complete pointwise feasible set.  Only after that
pointwise analysis do we pass from pointwise thresholds to the robust value.
The remaining maximization is split into radial and angular parts.  Finally,
the one-dimensional case is treated separately.

\begin{lemma}[Scale normalization]
\Label{lem:S-normalization}
Let $M:=\|\mathbf{m}\|_2$, $\widehat{\mathbf{m}}:=\mathbf{m}/M$, and
$\rho:=R/M$.  Then
\begin{equation}
\alpha_{\mathrm{ball}}(\mathbf{m},R,c)
=
\alpha_{\mathrm{ball}}(\widehat{\mathbf{m}},\rho,c).
\Label{eq:normalization}
\end{equation}
Moreover, $0<\rho<c$, $\|\widehat{\mathbf{m}}\|_2=1$, and
$\mathbb{B}(\widehat{\mathbf{m}},\rho)$ does not contain the origin.
\end{lemma}

\begin{proof}
Since $\mathbf{m}\ne0$, we have $M>0$.  The map
$\mathbf{u}\mapsto\mathbf{x}:=\mathbf{u}/M$ is a bijection from
$\mathbb{B}(\mathbf{m},R)$ onto
$\mathbb{B}(\widehat{\mathbf{m}},\rho)$.  For corresponding points,
\eqref{eq:cone-positive-scaling} gives
\(
\mathcal{C}(\boldsymbol{\xi},c)=\mathcal{C}(\mathbf{x},c).
\)
Moreover, the relation
\(
\langle\alpha\mathbf{m}-\boldsymbol{\xi},\mathbf{y}\rangle
=M\langle\alpha\widehat{\mathbf{m}}-
\mathbf{x},\mathbf{y}\rangle
\)
holds for every $\mathbf{y}\in\mathbb{R}^d$.
Because $M>0$, a value $\alpha>0$ is feasible for the original robust
problem if and only if it is feasible for the normalized robust problem.
The two feasible sets are therefore identical, and taking their infima proves
\eqref{eq:normalization}.

The assumption $0<R<cM$ gives $0<\rho<c\le1$.  Also,
$\|\widehat{\mathbf{m}}\|_2=1$ and $\rho<1$, so the distance from the center
of the normalized ball to the origin is larger than its radius.  Hence the
normalized ball does not contain the origin.
\qed
\end{proof}

For the remainder of the analysis in dimensions $d\ge2$, we work entirely
in the normalized variables supplied by Lemma~\ref{lem:S-normalization}.
To avoid carrying hats through the geometric argument, we henceforth write
$\mathbf{m}$ for the normalized center $\widehat{\mathbf{m}}$.  A point of
the normalized ball is denoted by $\mathbf{x}$; the corresponding original
point, when it must be mentioned, is denoted by $\boldsymbol{\xi}$.
Thus
\begin{equation}
\|\mathbf{m}\|_2=1,
\qquad
0<\rho<c\le1,
\qquad
s:=\sqrt{1-c^2}.
\Label{eq:normalized-common-parameters}
\end{equation}
For every unit vector $\mathbf{u}$, define
\begin{equation}
p(\mathbf{u}):=\langle\mathbf{m},\mathbf{u}\rangle.
\Label{eq:p-definition}
\end{equation}
Quantities whose definitions require additional
conditions, including $z_\pm(p)$ and $\alpha_\pm$, will be introduced only
after those conditions have been established.

\subsection{Intersection of fixed rays with the normalized ball} \Label{S-fixed-ray-intersection}

We first describe the intersection of the normalized ball with a fixed
ray from the origin. 
The following lemma characterizes both the directions whose rays
intersect the ball and the admissible radial coordinates along each
such direction.

\begin{lemma}[Intersection of a fixed ray with the normalized ball]
\Label{lem:S-ray-ball-intersection}
Let $\|\mathbf{m}\|_2=1$, $0<\rho<1$, and let
$\mathbf{u}\in\mathbb{R}^d$ be a unit vector. 
Given
\(p:=p(\mathbf{u}) \),
there exists $z>0$ such that
$z\mathbf{u}\in\mathbb{B}(\mathbf{m},\rho)$ if and only if
\begin{equation}
\sqrt{1-\rho^2}\le p\le1.
\Label{eq:p-range}
\end{equation}
Whenever \eqref{eq:p-range} holds, define
\[
z_-(p)
:=
p-\sqrt{p^2-(1-\rho^2)},
\qquad
z_+(p)
:=
p+\sqrt{p^2-(1-\rho^2)}.
\]
Then both $z_-(p)$ and $z_+(p)$ are strictly positive, and, for every
$z>0$,
\begin{equation}
z\mathbf{u}\in\mathbb{B}(\mathbf{m},\rho)
\quad\Longleftrightarrow\quad
z_-(p)\le z\le z_+(p).
\Label{eq:radial-interval}
\end{equation}
\end{lemma}

\begin{proof}
For $z>0$, the condition
$z\mathbf{u}\in\mathbb{B}(\mathbf{m},\rho)$ is equivalent to
\(
\|z\mathbf{u}-\mathbf{m}\|_2^2\le\rho^2.
\)
Since $\|\mathbf{m}\|_2=\|\mathbf{u}\|_2=1$ and
$p=\langle\mathbf{m},\mathbf{u}\rangle$, this condition becomes
\begin{equation}
z^2-2pz+1-\rho^2\le0.
\Label{eq:ray-quadratic}
\end{equation}
The corresponding quadratic equation
\begin{equation}
z^2-2pz+1-\rho^2=0
\Label{eq:ray-boundary-equation}
\end{equation}
has real roots if and only if
\[
p^2\ge1-\rho^2.
\]

Suppose first that a positive $z$ satisfies
\eqref{eq:ray-quadratic}. Then, we have
\(
2pz
\ge
z^2+1-\rho^2
>
0,
\)
so $p>0$. Combining this inequality with
$p^2\ge1-\rho^2$ gives
\(
p\ge\sqrt{1-\rho^2}.
\)
Moreover, $p\le1$ follows from the Cauchy--Schwarz inequality. This
proves the necessity of \eqref{eq:p-range}.

Conversely, suppose that \eqref{eq:p-range} holds. The roots of
\eqref{eq:ray-boundary-equation} are $z_-(p)$ and $z_+(p)$ as defined
above. Their sum and product satisfy
\[
z_-(p)+z_+(p)=2p>0,
\qquad
z_-(p)z_+(p)=1-\rho^2>0.
\]
Hence both roots are strictly positive. Since the leading coefficient
in \eqref{eq:ray-quadratic} is positive, the quadratic expression is
nonpositive exactly between its two roots. Therefore,
\eqref{eq:ray-quadratic} holds exactly when
\(
z_-(p)\le z\le z_+(p).
\)
This proves \eqref{eq:radial-interval} and completes the proof.
\qed
\end{proof}

\subsection{Pointwise feasibility at a fixed point} \Label{S-pointwise-feasibility}

We now turn to the pointwise dominance problem at a fixed point
$\mathbf{x}\in\mathbb{B}(\mathbf{m},\rho)$.  Since the normalized ball
does not contain the origin, define
\[
z:=\|\mathbf{x}\|_2>0,
\qquad
\mathbf{u}:=\frac{\mathbf{x}}{\|\mathbf{x}\|_2},
\qquad
\mathbf{x}=z\mathbf{u},
\qquad
\|\mathbf{u}\|_2=1.
\]
By \eqref{eq:cone-positive-scaling},
\begin{equation}
\mathcal{C}(\mathbf{x},c)=\mathcal{C}(\mathbf{u},c).
\Label{eq:cone-x-u}
\end{equation}
The pointwise value in \eqref{eq:pointwise-alpha} is determined by the
condition that $\alpha\mathbf{m}-\mathbf{x}$ have a nonnegative inner
product with every vector in $\mathcal{C}(\mathbf{x},c)$.  By
\eqref{eq:cone-x-u}, the pointwise problem can be expressed as a membership
condition in the dual cone $\mathcal{C}(\mathbf{u},c)^*$.  We first identify
this dual cone.
\begin{lemma}[Dual of the circular cone]
\Label{lem:S-dual-circular-cone}
Let $d\ge2$, let $\mathbf{u}\in\mathbb{R}^d$ be a unit vector, and let
$c\in(0,1]$.  Set $s:=\sqrt{1-c^2}$.  Then, we have
\begin{equation}
\mathcal{C}(\mathbf{u},c)^*
=
\left\{
\mathbf{w}\in\mathbb{R}^d:
\langle\mathbf{u},\mathbf{w}\rangle
\ge s\|\mathbf{w}\|_2
\right\}.
\Label{eq:dual-circular-cone}
\end{equation}
\end{lemma}

\begin{proof}
For $\mathbf{w}\in\mathbb{R}^d$, define
\(
a:=\langle\mathbf{u},\mathbf{w}\rangle\) and 
\(\mathbf{w}_{\perp}:=\mathbf{w}-a\mathbf{u}.
\)
Then, we have
\[
\mathbf{w}=a\mathbf{u}+\mathbf{w}_{\perp},
\qquad
\langle\mathbf{u},\mathbf{w}_{\perp}\rangle=0,
\qquad
\|\mathbf{w}\|_2^2=a^2+\|\mathbf{w}_{\perp}\|_2^2.
\]
We first express membership in $\mathcal{C}(\mathbf{u},c)^*$ in terms of
$a$ and $\mathbf{w}_{\perp}$.  Since $\mathcal{C}(\mathbf{u},c)$ is
homogeneous, it is enough to test the inner product against unit vectors in
the cone.  Such a unit vector can be written as
\[
\mathbf{v}=q\mathbf{u}+\mathbf{v}_{\perp},
\qquad
q\in[c,1],
\qquad
\langle\mathbf{u},\mathbf{v}_{\perp}\rangle=0,
\qquad
\|\mathbf{v}_{\perp}\|_2=\sqrt{1-q^2}.
\]
Taking $\mathbf{v}=\mathbf{u}$ shows that
$\mathbf{w}\in\mathcal{C}(\mathbf{u},c)^*$ requires $a\ge0$.
For fixed $q$, the Cauchy--Schwarz inequality gives
\[
\langle\mathbf{w},\mathbf{v}\rangle
=aq+\langle\mathbf{w}_{\perp},\mathbf{v}_{\perp}\rangle
\ge aq-\|\mathbf{w}_{\perp}\|_2\sqrt{1-q^2}.
\]
Because $d\ge2$, equality can be attained by choosing
$\mathbf{v}_{\perp}$ in the direction opposite to
$\mathbf{w}_{\perp}$, with the case $\mathbf{w}_{\perp}=0$ being
immediate.  When $a\ge0$, the right-hand side is nondecreasing in
$q\in[c,1]$.  Its minimum is therefore
$ac-s\|\mathbf{w}_{\perp}\|_2$.  Hence, the following 
equivalence condition holds:
\begin{equation}
\mathbf{w}\in\mathcal{C}(\mathbf{u},c)^*
\quad\Longleftrightarrow\quad
a\ge0
\quad\text{and}\quad
ac\ge s\|\mathbf{w}_{\perp}\|_2.
\Label{eq:dual-two-conditions}
\end{equation}

We now show that the two conditions on the right-hand side of
\eqref{eq:dual-two-conditions} are equivalent to the inequality on the
right-hand side of \eqref{eq:dual-circular-cone}.  Suppose first that
$a\ge0$ and $ac\ge s\|\mathbf{w}_{\perp}\|_2$.  Squaring the second
inequality and using $c^2=1-s^2$ gives
\[
a^2
\ge
s^2\bigl(a^2+\|\mathbf{w}_{\perp}\|_2^2\bigr)
=
s^2\|\mathbf{w}\|_2^2.
\]
Since $a\ge0$, we obtain $a\ge s\|\mathbf{w}\|_2$.

Conversely, suppose that $a\ge s\|\mathbf{w}\|_2$.  Then $a\ge0$, and
squaring gives
\(
c^2a^2\ge s^2\|\mathbf{w}_{\perp}\|_2^2.
\)
All quantities are nonnegative, so
$ac\ge s\|\mathbf{w}_{\perp}\|_2$.  Thus the two conditions in
\eqref{eq:dual-two-conditions} hold.  This proves
\eqref{eq:dual-circular-cone}.
\qed
\end{proof}

We now return to the pointwise dominance problem for the fixed point
$\mathbf{x}=z\mathbf{u}$ introduced above, and retain the notation
$p=p(\mathbf{u})$.  The algebraic boundary equation that arises in the
analysis below is
\begin{equation}
0=
\bigl((cp-s\sqrt{1-p^2})\alpha-zc\bigr)
\bigl((cp+s\sqrt{1-p^2})\alpha-zc\bigr).
\Label{eq:pointwise-boundary-equation}
\end{equation}
Its two formal solutions are
\begin{equation}
\alpha_-:=
\frac{zc}{cp+s\sqrt{1-p^2}},
\qquad
\alpha_+:=
\frac{zc}{cp-s\sqrt{1-p^2}}.
\Label{eq:pointwise-two-roots}
\end{equation}
The following lemma proves that the two denominators in
\eqref{eq:pointwise-two-roots} are strictly positive and that
$\alpha_+$ is the pointwise dominance threshold.
\begin{lemma}[Pointwise dominance threshold]
\Label{lem:S-pointwise-scaling}
Under the normalized assumptions
\eqref{eq:normalized-common-parameters} and the notation above, the
quantities $\alpha_-$ and $\alpha_+$ in
\eqref{eq:pointwise-two-roots} are well defined and satisfy
$0<\alpha_-\le\alpha_+$.  Moreover, a positive number $\alpha$
satisfies
\begin{equation}
\langle\alpha\mathbf{m}-\mathbf{x},\mathbf{y}\rangle\ge0
\qquad
\text{for every }\mathbf{y}\in\mathcal{C}(\mathbf{x},c)
\Label{eq:pointwise-dominance-condition}
\end{equation}
if and only if
\begin{equation}
\alpha\ge\alpha_+.
\Label{eq:pointwise-feasibility-characterization}
\end{equation}
\end{lemma}
\begin{proof}
We first verify that the quantities in
\eqref{eq:pointwise-two-roots} are well defined.  Since
$\mathbf{x}=z\mathbf{u}\in\mathbb{B}(\mathbf{m},\rho)$ with $z>0$,
Lemma~\ref{lem:S-ray-ball-intersection} gives
\(
p\ge\sqrt{1-\rho^2}.
\)
The assumption $\rho<c$ and the identity $s=\sqrt{1-c^2}$ imply
\(
\sqrt{1-\rho^2}>\sqrt{1-c^2}=s.
\)
Hence $p>s\ge0$.  Using $c^2+s^2=1$, we obtain
\(
(cp)^2-s^2(1-p^2)=p^2-s^2>0.
\)
Both $cp$ and $s\sqrt{1-p^2}$ are nonnegative, so
\begin{equation}
cp-s\sqrt{1-p^2}>0.
\Label{eq:pointwise-denominator-positive}
\end{equation}
Moreover, $p>0$ and $c>0$, and therefore
$cp+s\sqrt{1-p^2}>0$.  Thus $\alpha_-$ and $\alpha_+$ are positive
and well defined.  Since the denominator of $\alpha_-$ is at least
the denominator of $\alpha_+$, we also have
$\alpha_-\le\alpha_+$.

We now characterize pointwise feasibility.
By \eqref{eq:cone-x-u} and the definition of the dual cone,
\eqref{eq:pointwise-dominance-condition} is equivalent to
\(
\alpha\mathbf{m}-z\mathbf{u}
\in\mathcal{C}(\mathbf{u},c)^*.
\)
Applying Lemma~\ref{lem:S-dual-circular-cone} to
$\mathbf{w}=\alpha\mathbf{m}-z\mathbf{u}$ shows that this membership
condition is equivalent to
\begin{equation}
\alpha p-z
\ge
s\sqrt{\alpha^2-2\alpha zp+z^2}
\Label{eq:pointwise-soc}
\end{equation}
because 
\(
\langle\mathbf{u},\alpha\mathbf{m}-z\mathbf{u}\rangle
=\alpha p-z
\)
and
\(
\|\alpha\mathbf{m}-z\mathbf{u}\|_2^2
=\alpha^2-2\alpha zp+z^2.
\)

Since the right-hand side of \eqref{eq:pointwise-soc} is
nonnegative, \eqref{eq:pointwise-soc} is equivalent to the
conjunction of the sign condition
\begin{equation}
\alpha p-z\ge0
\Label{eq:pointwise-sign-condition}
\end{equation}
and the quadratic inequality
\begin{equation}
(\alpha p-z)^2-s^2(\alpha^2-2\alpha zp+z^2)\ge0.
\Label{eq:pointwise-quadratic-inequality}
\end{equation}
The quadratic expression in
\eqref{eq:pointwise-quadratic-inequality} factors as
\begin{align}
&(\alpha p-z)^2-s^2(\alpha^2-2\alpha zp+z^2)
\notag\\
&=
\bigl((cp-s\sqrt{1-p^2})\alpha-zc\bigr)
\bigl((cp+s\sqrt{1-p^2})\alpha-zc\bigr).
\Label{eq:pointwise-factorization}
\end{align}
Thus its boundary equation is precisely
\eqref{eq:pointwise-boundary-equation}, whose roots are
$\alpha_-$ and $\alpha_+$.  Since its leading coefficient is
$p^2-s^2>0$, the quadratic inequality holds precisely when
\[
0<\alpha\le\alpha_-
\qquad\text{or}\qquad
\alpha\ge\alpha_+.
\]

It remains to impose the sign condition
\eqref{eq:pointwise-sign-condition}.  
First, we show that the lower branch
\(
0<\alpha\le\alpha_-
\)
does not satisfy the sign condition
\eqref{eq:pointwise-sign-condition}, except in the degenerate case
where the two roots coincide. Since $p>0$, the function
\(
\alpha\mapsto\alpha p-z
\)
is increasing. Therefore, its largest value on the lower branch is
attained at $\alpha=\alpha_-$. Substituting the definition of
$\alpha_-$ gives
\begin{align}
\alpha_-p-z
=
\frac{zcp}{cp+s\sqrt{1-p^2}}-z
=
-\frac{zs\sqrt{1-p^2}}
{cp+s\sqrt{1-p^2}}
\le0.
\Label{eq:lower-branch-sign}
\end{align}
If $\alpha_-<\alpha_+$, then
$s\sqrt{1-p^2}>0$, and hence
\(
\alpha_-p-z<0.
\)
Consequently, every $0<\alpha\le\alpha_-$ satisfies
\(
\alpha p-z
\le
\alpha_-p-z
<0,
\)
which violates \eqref{eq:pointwise-sign-condition}.
If $\alpha_-=\alpha_+$, their common value already belongs to the
upper branch $\alpha\ge\alpha_+$.

Conversely, if $\alpha\ge\alpha_+$, then
\[
\alpha p-z
\ge
\alpha_+p-z
=
\frac{zs\sqrt{1-p^2}}
{cp-s\sqrt{1-p^2}}
\ge0.
\]
Hence the sign condition holds, and the quadratic inequality also
holds because $\alpha$ lies on or to the right of its larger root.
The two conditions equivalent to \eqref{eq:pointwise-soc} are
therefore satisfied.  This proves that
\eqref{eq:pointwise-dominance-condition} holds if and only if
$\alpha\ge\alpha_+$.
\qed
\end{proof}
The pointwise dominance value can now be read directly from the equivalence
proved in Lemma~\ref{lem:S-pointwise-scaling}.  Indeed,
\eqref{eq:pointwise-dominance-condition} holds exactly when
$\alpha\ge\alpha_+$.  Hence the definition \eqref{eq:pointwise-alpha} and
\eqref{eq:pointwise-two-roots} give
\begin{equation}
\alpha_{\mathrm{pt}}(\mathbf{x};\mathbf{m},c)
=
\alpha_+
=
\frac{zc}{cp-s\sqrt{1-p^2}}.
\Label{eq:pointwise-formula}
\end{equation}

\subsection{From pointwise feasibility to a common robust scaling}

The purpose of this subsection is to pass from the scaling required at one
fixed point to a common scaling that works throughout the normalized ball.
We first express the robust ball value as the supremum of the pointwise
dominance values.  We then reduce this supremum to a one-variable
maximization by maximizing along each fixed direction and replacing the
remaining choice of direction by the scalar parameter $p=p(\mathbf{u})$.

For each fixed point $\mathbf{x}\in\mathbb{B}(\mathbf{m},\rho)$, define
\[
\mathcal{A}_{\mathrm{pt}}(\mathbf{x})
:=
\left\{
\alpha>0:
\langle\alpha\mathbf{m}-\mathbf{x},\mathbf{y}\rangle\ge0
\text{ for every }\mathbf{y}\in\mathcal{C}(\mathbf{x},c)
\right\}.
\]
Lemma~\ref{lem:S-pointwise-scaling} and
\eqref{eq:pointwise-formula} give
\[
\mathcal{A}_{\mathrm{pt}}(\mathbf{x})
=[\alpha_{\mathrm{pt}}(\mathbf{x};\mathbf{m},c),\infty).
\]
We now distinguish this local set from the set of scalings that work
simultaneously at every point of the ball.  Define
\[
\mathcal{A}_{\mathrm{ball}}
:=
\left\{
\alpha>0:
\alpha\in\mathcal{A}_{\mathrm{pt}}(\mathbf{x})
\text{ for every }\mathbf{x}\in\mathbb{B}(\mathbf{m},\rho)
\right\}.
\]
Thus $\mathcal{A}_{\mathrm{ball}}$ is precisely the feasible set in the
definition of $\alpha_{\mathrm{ball}}(\mathbf{m},\rho,c)$.  Since a common
scaling must belong to every pointwise feasible set,
\begin{align}
\mathcal{A}_{\mathrm{ball}}
=
\bigcap_{\mathbf{x}\in\mathbb{B}(\mathbf{m},\rho)}
\mathcal{A}_{\mathrm{pt}}(\mathbf{x})
=
\left[
\sup_{\mathbf{x}\in\mathbb{B}(\mathbf{m},\rho)}
\alpha_{\mathrm{pt}}(\mathbf{x};\mathbf{m},c),
\infty
\right].
\Label{eq:robust-feasible-set}
\end{align}
Taking the infimum gives
\begin{equation}
\alpha_{\mathrm{ball}}(\mathbf{m},\rho,c)
=
\sup_{\mathbf{x}\in\mathbb{B}(\mathbf{m},\rho)}
\alpha_{\mathrm{pt}}(\mathbf{x};\mathbf{m},c).
\Label{eq:robust-pointwise-representation}
\end{equation}

We next reduce the supremum in
\eqref{eq:robust-pointwise-representation} to a one-variable maximization.

\begin{lemma}[Reduction to a one-variable maximization]
\Label{lem:S-reduction-to-p}
Under the normalized assumptions
\eqref{eq:normalized-common-parameters},
\begin{equation}
\sup_{\mathbf{x}\in\mathbb{B}(\mathbf{m},\rho)}
\alpha_{\mathrm{pt}}(\mathbf{x};\mathbf{m},c)
=
\max_{\sqrt{1-\rho^2}\le p\le1}
\frac{c\left(p+\sqrt{p^2-(1-\rho^2)}\right)}
{cp-s\sqrt{1-p^2}}.
\Label{eq:p-objective}
\end{equation}
\end{lemma}

\begin{proof}
Every point of $\mathbb{B}(\mathbf{m},\rho)$ is nonzero and therefore has a
unique representation $\mathbf{x}=z\mathbf{u}$ with $z>0$ and
$\|\mathbf{u}\|_2=1$.  
Thus the ball is partitioned by the nonempty intersections of the rays $\{z\mathbf{u}:z>0\}$ with $\mathbb{B}(\mathbf{m},\rho)$. For each such direction $\mathbf{u}$, the admissible radial coordinates form the compact interval given in \eqref{eq:radial-interval}, and $\alpha_{\mathrm{pt}}(z\mathbf{u};\mathbf{m},c)$ is continuous in $z$. Hence the supremum over the ball can be obtained by first maximizing over the admissible radial coordinates on each ray and then taking the supremum over all directions whose rays intersect the ball as
\begin{align}
\sup_{\mathbf{x}\in\mathbb{B}(\mathbf{m},\rho)}
\alpha_{\mathrm{pt}}(\mathbf{x};\mathbf{m},c)
=
\sup_{\substack{\|\mathbf{u}\|_2=1\\
\exists z>0:\ z\mathbf{u}\in\mathbb{B}(\mathbf{m},\rho)}}
\ \max_{\substack{z>0\\
z\mathbf{u}\in\mathbb{B}(\mathbf{m},\rho)}}
\alpha_{\mathrm{pt}}(z\mathbf{u};\mathbf{m},c).
\Label{eq:direction-radial-decomposition}
\end{align}

Fix a unit vector $\mathbf{u}$ occurring in the outer supremum and write
$p:=p(\mathbf{u})$.  Formula \eqref{eq:pointwise-formula} gives
\begin{equation}
\alpha_{\mathrm{pt}}(z\mathbf{u};\mathbf{m},c)
=
z\frac{c}{cp-s\sqrt{1-p^2}}.
\Label{eq:pointwise-fixed-direction}
\end{equation}
By \eqref{eq:radial-interval}, the admissible radial coordinates form
the interval $[z_-(p),z_+(p)]$. 
Since the coefficient of $z$ in \eqref{eq:pointwise-fixed-direction} is strictly positive by \eqref{eq:pointwise-denominator-positive}, $\alpha_{\mathrm{pt}}(z\mathbf{u};\mathbf{m},c)$ is strictly increasing in $z$. Its maximum is therefore attained at the upper endpoint $z=z_+(p)$. Substituting the definition of $z_+(p)$ into \eqref{eq:pointwise-fixed-direction} gives
\begin{equation}
\max_{\substack{z>0\\
z\mathbf{u}\in\mathbb{B}(\mathbf{m},\rho)}}
\alpha_{\mathrm{pt}}(z\mathbf{u};\mathbf{m},c)
=
\frac{c\left(p+\sqrt{p^2-(1-\rho^2)}\right)}
{cp-s\sqrt{1-p^2}}.
\Label{eq:fixed-direction-maximum}
\end{equation}

It remains to replace the outer choice of $\mathbf{u}$ in
\eqref{eq:direction-radial-decomposition} by the corresponding choice of
$p=p(\mathbf{u})$.  Equation \eqref{eq:p-range} gives one inclusion in
\begin{equation}
\left\{
 p(\mathbf{u}):
 \|\mathbf{u}\|_2=1,
 \ \exists z>0:\ z\mathbf{u}\in\mathbb{B}(\mathbf{m},\rho)
\right\}
=
\left[\sqrt{1-\rho^2},1\right].
\Label{eq:p-attainable-set}
\end{equation}
For the reverse inclusion, choose a unit vector $\mathbf{v}$ orthogonal to
$\mathbf{m}$.  For every
$p\in[\sqrt{1-\rho^2},1]$, the vector
\(
\mathbf{u}
=
p\mathbf{m}+\sqrt{1-p^2}\,\mathbf{v}
\)
is a unit vector satisfying $p(\mathbf{u})=p$, and
\eqref{eq:p-range} shows that its ray meets the ball.  This proves
\eqref{eq:p-attainable-set}.  Substituting
\eqref{eq:fixed-direction-maximum} and \eqref{eq:p-attainable-set} into
\eqref{eq:direction-radial-decomposition} yields
\eqref{eq:p-objective}.  The maximum is attained because its objective is
continuous on the compact interval $[\sqrt{1-\rho^2},1]$ and its denominator
is positive by \eqref{eq:pointwise-denominator-positive}.
\qed
\end{proof}

The remaining problem is the one-variable maximization on the right-hand
side of \eqref{eq:p-objective}.  We evaluate it by introducing an angular
parametrization of $p$.

Choose $\phi\in(0,\pi/2)$ and $\psi\in[0,\pi/2)$ such that
\[
\rho=\sin\phi,
\qquad
c=\cos\psi,
\qquad
s=\sin\psi.
\]
Since $\rho<c$, we have $\phi+\psi<\pi/2$.  Under the substitution
$p=\cos\theta$, the interval
$\sqrt{1-\rho^2}\le p\le1$ becomes $0\le\theta\le\phi$, and the objective
in \eqref{eq:p-objective} becomes the function defined below.

\begin{lemma}[One-variable angular maximization]
\Label{lem:S-one-variable-angular-maximum}
Let $0<\rho<c\le1$, set
\[
s:=\sqrt{1-c^2},
\qquad
\phi:=\arcsin\rho,
\qquad
\psi:=\arccos c,
\]
and define
\begin{equation}
F(\theta)
:=
\frac{c\left(
\cos\theta+\sqrt{\rho^2-\sin^2\theta}
\right)}
{\cos(\theta+\psi)},
\qquad
0\le\theta\le\phi.
\Label{eq:angular-objective}
\end{equation}
Then
\begin{equation}
\max_{0\le\theta\le\phi}F(\theta)
=
\frac{c(1-\rho^2)}{c-\rho}.
\Label{eq:one-variable-angular-maximum}
\end{equation}
If $c<1$, the maximum is attained at the unique point
$\theta_*\in(0,\phi)$ satisfying
\begin{equation}
\tan\theta_*
=
\frac{\rho s}{1-\rho c}.
\Label{eq:theta-star-supp}
\end{equation}
If $c=1$, the maximum is attained at $\theta_*=0$.
\end{lemma}

\begin{proof}
The proof is divided into three steps.  In Step 1, we introduce an auxiliary
angle $\beta(\theta)$ and show that the sign of $F'(\theta)$ is determined by
the comparison between $\beta(\theta)$ and $\theta+\psi$.  In Step 2, for
$c<1$, we prove that the difference $\beta(\theta)-\theta$ is strictly
increasing.  This identifies the unique maximizer and provides the relation
needed to evaluate the maximum.  In Step 3, we use the same derivative-sign
criterion to treat the endpoint case $c=1$.

\noindent\textbf{Step 1: Reduction to a comparison of two angles.}
The objective of this step is to prove that, for $0<\theta<\phi$, the
derivative of $F$ has the same sign as
\begin{equation}
\theta+\psi-\beta(\theta),
\Label{eq:angular-derivative-sign-target}
\end{equation}
where the auxiliary angle $\beta(\theta)$ will be defined below.  This
reduction replaces the direct analysis of the derivative by a comparison of
two angles.

Set
\begin{equation}
A(\theta):=\sqrt{\rho^2-\sin^2\theta},
\qquad 0\le\theta\le\phi.
\Label{eq:angular-A-definition}
\end{equation}
For $0\le\theta<\phi$, define $\beta(\theta)\in[0,\pi/2)$ by
\begin{equation}
\sin\beta(\theta)=\frac{\sin\theta}{\rho},
\qquad
\cos\beta(\theta)=\frac{A(\theta)}{\rho}.
\Label{eq:angular-beta-definition}
\end{equation}
To verify that this definition is well posed, equation \eqref{eq:angular-A-definition} gives 
\begin{equation} \left(\frac{\sin\theta}{\rho}\right)^2 + \left(\frac{A(\theta)}{\rho}\right)^2 = 1. \Label{eq:angular-beta-unit-circle} \end{equation} 
Moreover, $\sin\theta\ge0$ and $A(\theta)>0$ for $0\le\theta<\phi$. Hence the two quantities on the right-hand side of \eqref{eq:angular-beta-definition} are the sine and cosine of a unique angle $\beta(\theta)\in[0,\pi/2)$.
In particular, \eqref{eq:angular-beta-definition} gives
\begin{equation}
\frac{\sin\theta}{A(\theta)}=\tan\beta(\theta).
\Label{eq:angular-beta-tangent}
\end{equation}

We now express the derivative of $F$ in terms of this angle.  
The relation
\eqref{eq:angular-A-definition} implies
\(
A'(\theta)=-\frac{\sin\theta\cos\theta}{A(\theta)}\)
for $0<\theta<\phi$,
which yields
\begin{align}
\frac{d}{d\theta}
\log\bigl(\cos\theta+A(\theta)\bigr)
=
\frac{-\sin\theta+A'(\theta)}
{\cos\theta+A(\theta)}
=-\frac{\sin\theta}{A(\theta)}.
\Label{eq:angular-numerator-log-derivative}
\end{align}
Applying \eqref{eq:angular-numerator-log-derivative} to the definition
\eqref{eq:angular-objective} and then using
\eqref{eq:angular-beta-tangent}, we obtain
\begin{align}
\frac{d}{d\theta}\log F(\theta)
=
-\frac{\sin\theta}{A(\theta)}
+\tan(\theta+\psi)
=
\tan(\theta+\psi)-\tan\beta(\theta).
\Label{eq:angular-derivative-angle-comparison}
\end{align}
For $0<\theta<\phi$, both $\theta+\psi$ and $\beta(\theta)$ belong to
$(0,\pi/2)$.  Since the tangent function is strictly increasing on this
interval, \eqref{eq:angular-derivative-angle-comparison} shows that the
derivative of $F$ has the same sign as
\begin{equation}
\theta+\psi-\beta(\theta).
\Label{eq:angular-derivative-sign}
\end{equation}

\noindent\textbf{Step 2: The unique maximizer when $c<1$.}
Assume that $c<1$, or equivalently that $\psi>0$.  The objective of this step
is first to prove that there exists a unique $\theta_*\in(0,\phi)$ satisfying
\begin{equation}
\beta(\theta_*)=\theta_*+\psi,
\Label{eq:angular-critical-angle-target}
\end{equation}
and then to evaluate $F(\theta_*)$.  By the derivative-sign criterion proved
in Step 1, the first part will show that $F$ is increasing before
$\theta_*$ and decreasing after $\theta_*$.  Hence it will also show that
$\theta_*$ is the unique global maximizer.

Define
\begin{equation}
G(\theta):=\beta(\theta)-\theta,
\qquad 0\le\theta<\phi.
\Label{eq:angular-G-definition}
\end{equation}
We first show that $G$ is strictly increasing.  Differentiating the first
identity in \eqref{eq:angular-beta-definition} gives
\(
\cos\theta
=
\rho\cos\beta(\theta)\,\beta'(\theta).
\)
Combining this relation with the second identity in \eqref{eq:angular-beta-definition}, we therefore
obtain
\begin{equation}
\beta'(\theta)=\frac{\cos\theta}{A(\theta)}.
\Label{eq:angular-beta-derivative}
\end{equation}
Moreover, \eqref{eq:angular-A-definition} and $0<\rho<1$ give
\begin{align}
A(\theta)^2
=
\rho^2-\sin^2\theta
<
1-\sin^2\theta
=
\cos^2\theta.\label{VBWE}
\end{align}
Hence \eqref{eq:angular-beta-derivative} implies
$\beta'(\theta)>1$, and thus $G$ is strictly increasing on
$[0,\phi)$.

The definition \eqref{eq:angular-beta-definition} gives
\begin{equation}
G(0)=0,
\qquad
\lim_{\theta\uparrow\phi}G(\theta)
=
\frac{\pi}{2}-\phi.
\Label{eq:angular-G-endpoints}
\end{equation}
Since $\rho<c$, we have $\phi+\psi<\pi/2$, and therefore
\begin{equation}
0<\psi<\frac{\pi}{2}-\phi.
\Label{eq:angular-psi-range}
\end{equation}
The strict monotonicity of $G$, together with
\eqref{eq:angular-G-endpoints} and \eqref{eq:angular-psi-range}, shows that
there exists a unique $\theta_*\in(0,\phi)$ satisfying
\begin{equation}
G(\theta_*)=\psi,
\qquad\text{or equivalently}\qquad
\beta(\theta_*)=\theta_*+\psi.
\Label{eq:angular-critical-angle}
\end{equation}
By \eqref{eq:angular-G-definition}, the expression in
\eqref{eq:angular-derivative-sign} equals $\psi-G(\theta)$.  Hence it is
positive for $0<\theta<\theta_*$ and negative for
$\theta_*<\theta<\phi$.  Equation
\eqref{eq:angular-derivative-angle-comparison} therefore shows that $F$ is
strictly increasing before $\theta_*$ and strictly decreasing after
$\theta_*$.
Thus $\theta_*$ is the unique global maximizer of $F$.
We have now proved the existence and uniqueness of the maximizer.  It remains
to evaluate $F$ at this point.  For this purpose, we use the critical-point
relation \eqref{eq:angular-critical-angle}.

Substituting \eqref{eq:angular-critical-angle} into
\eqref{eq:angular-beta-definition} gives
\begin{equation}
\sin\theta_*=\rho\sin(\theta_*+\psi),
\qquad
A(\theta_*)=\rho\cos(\theta_*+\psi).
\Label{eq:angular-critical-relations}
\end{equation}
Using $c=\cos\psi$ and $s=\sin\psi$, the first identity in
\eqref{eq:angular-critical-relations} and the angle-addition formula give
\begin{align}
\sin\theta_*
=\rho\sin(\theta_*+\psi)
=\rho c\sin\theta_*+\rho s\cos\theta_*.
\Label{eq:angular-critical-sine-expansion}
\end{align}
Moving the term $\rho c\sin\theta_*$ to the left-hand side of
\eqref{eq:angular-critical-sine-expansion}, we obtain
\begin{equation}
(1-\rho c)\sin\theta_*
=
\rho s\cos\theta_*.
\Label{eq:angular-critical-sine-rearranged}
\end{equation}
Since $1-\rho c>0$ and $\cos\theta_*>0$, we may divide
\eqref{eq:angular-critical-sine-rearranged} by
$(1-\rho c)\cos\theta_*$.  This gives
\begin{equation}
\tan\theta_*
=
\frac{\rho s}{1-\rho c}.
\Label{eq:theta-star-supp-again}
\end{equation}
Thus \eqref{eq:theta-star-supp} holds.

We next use \eqref{eq:theta-star-supp-again} to evaluate the remaining
ratio. The combination of
\eqref{eq:theta-star-supp-again} and 
the angle-addition formula
implies
\begin{align}
&\frac{\cos\theta_*}{\cos(\theta_*+\psi)}
=
\frac{\cos\theta_*}
{c\cos\theta_*-s\sin\theta_*}
=
\frac{1}{c-s\tan\theta_*}
=
\frac{1}
{c-\dfrac{\rho s^2}{1-\rho c}}
\notag\\
&=
\frac{1-\rho c}
{c(1-\rho c)-\rho s^2}
=
\frac{1-\rho c}{c-\rho},
\Label{eq:angular-critical-cosine-ratio}
\end{align}
where the last equality follows from $c^2+s^2=1$.
Dividing the second identity in
\eqref{eq:angular-critical-relations} by
$\cos(\theta_*+\psi)>0$ and then applying
\eqref{eq:angular-critical-cosine-ratio} to
\eqref{eq:angular-objective}, we obtain
\begin{align}
F(\theta_*)
=
c\left(
\frac{\cos\theta_*}{\cos(\theta_*+\psi)}+\rho
\right)
=
c\left(
\frac{1-\rho c}{c-\rho}+\rho
\right)
=
\frac{c(1-\rho^2)}{c-\rho}.
\Label{eq:angular-maximum-evaluation}
\end{align}
This proves \eqref{eq:one-variable-angular-maximum} when $c<1$.

\noindent\textbf{Step 3: The endpoint maximizer when $c=1$.}
It remains to treat the case $c=1$, which is equivalent to $\psi=0$.  The
objective of this step is to use the derivative-sign criterion from Step 1,
together with the monotonicity calculation for $G$ in Step 2, to show that
$F$ is strictly decreasing on $(0,\phi)$.  This will imply that the maximum
is attained at the endpoint $\theta_*=0$.

The definition of $G$ in \eqref{eq:angular-G-definition} and the derivative formula \eqref{eq:angular-beta-derivative} remain valid when $c=1$. Moreover, the relations
\eqref{VBWE} and \eqref{eq:angular-beta-derivative} give 
\( \beta'(\theta)>1. \)
It follows from \eqref{eq:angular-G-definition} that 
\[ G'(\theta) = \beta'(\theta)-1 > 0. \] 
Thus $G$ is strictly increasing also when $c=1$. 
Since $G(0)=0$ by \eqref{eq:angular-G-endpoints}, we obtain 
\begin{equation} G(\theta)>0, \qquad 0<\theta<\phi. \Label{eq:angular-G-positive-c-one} 
\end{equation} 
Since $\psi=0$, the definition \eqref{eq:angular-G-definition} and \eqref{eq:angular-G-positive-c-one} give 
\( \theta+\psi-\beta(\theta) = -G(\theta) < 0. \)
Therefore, the derivative-sign criterion \eqref{eq:angular-derivative-angle-comparison} shows that 
\( F'(\theta)<0 \) for $0<\theta<\phi$. 

Hence, $F$ is maximized at
$\theta_*=0$.  Substitution into \eqref{eq:angular-objective} gives
\(
F(0)=1+\rho
=
\frac{1-\rho^2}{1-\rho}.
\)
This is the value in \eqref{eq:one-variable-angular-maximum} for $c=1$.
Together with Step 2, this proves the lemma.
\qed
\end{proof}

We now return to the normalized ball.  The preceding lemma evaluates the
one-variable objective obtained in Lemma~\ref{lem:S-reduction-to-p}; it
remains only to realize the maximizing parameters by a point of the ball.

\begin{lemma}[Angular maximum and attainment]
\Label{lem:S-angular-maximum}
Let $d\ge2$, $\|\mathbf{m}\|_2=1$, and $0<\rho<c\le1$.  
Then, the relation
\begin{equation}
\sup_{\mathbf{x}\in\mathbb{B}(\mathbf{m},\rho)}
\alpha_{\mathrm{pt}}(\mathbf{x};\mathbf{m},c)
=
\frac{c(1-\rho^2)}{c-\rho}
\Label{eq:normalized-ball-value}
\end{equation}
holds, and the supremum is attained.
\end{lemma}

\begin{proof}
Lemma~\ref{lem:S-reduction-to-p} expresses the required supremum as
\begin{equation}
\sup_{\mathbf{x}\in\mathbb{B}(\mathbf{m},\rho)}
\alpha_{\mathrm{pt}}(\mathbf{x};\mathbf{m},c)
=
\max_{\sqrt{1-\rho^2}\le p\le1}
\frac{c\left(p+\sqrt{p^2-(1-\rho^2)}\right)}
{cp-s\sqrt{1-p^2}}.
\Label{eq:angular-assembly-p-maximum}
\end{equation}
We now identify this scalar objective with the function studied in
Lemma~\ref{lem:S-one-variable-angular-maximum}.

Let $\phi:=\arcsin\rho$ and $\psi:=\arccos c$, and set $p=\cos\theta$.
Because $0<\phi<\pi/2$, the map $\theta\mapsto\cos\theta$ is a bijection
from $[0,\phi]$ onto $[\sqrt{1-\rho^2},1]$.  For corresponding values of
$p$ and $\theta$, we have $\sqrt{1-p^2}=\sin\theta$ and
$\sqrt{p^2-(1-\rho^2)}=\sqrt{\rho^2-\sin^2\theta}$.  The identity
$c\cos\theta-s\sin\theta=\cos(\theta+\psi)$ then shows that the objective
on the right-hand side of \eqref{eq:angular-assembly-p-maximum} is exactly
$F(\theta)$ from \eqref{eq:angular-objective}.  Consequently,
\begin{equation}
\sup_{\mathbf{x}\in\mathbb{B}(\mathbf{m},\rho)}
\alpha_{\mathrm{pt}}(\mathbf{x};\mathbf{m},c)
=
\max_{0\le\theta\le\phi}F(\theta).
\Label{eq:angular-assembly-theta-maximum}
\end{equation}
Lemma~\ref{lem:S-one-variable-angular-maximum} evaluates the right-hand
side as $c(1-\rho^2)/(c-\rho)$.  This proves
\eqref{eq:normalized-ball-value}.

It remains to show that the supremum is attained by a point of the normalized ball $\mathbb{B}(\mathbf{m},\rho)$. 
Let
$\theta_*$ be a maximizing angle supplied by
Lemma~\ref{lem:S-one-variable-angular-maximum}, and define
$p_*:=\cos\theta_*$.  
Since $0\le\theta_*\le\phi$, the value $p_*$ belongs
to $[\sqrt{1-\rho^2},1]$.  
Equation \eqref{eq:p-attainable-set} therefore
provides a unit vector $\mathbf{u}_*$ such that
$p(\mathbf{u}_*)=\langle\mathbf{m},\mathbf{u}_*\rangle=p_*$ and the ray
$\{z\mathbf{u}_*:z>0\}$ intersects $\mathbb{B}(\mathbf{m},\rho)$.

For the direction of $\mathbf{u}_*$, the fixed-direction maximum
\eqref{eq:fixed-direction-maximum} is attained at $z=z_+(p_*)$.  Define
\[
\mathbf{x}_*:=z_+(p_*)\mathbf{u}_*.
\]
Then $\mathbf{x}_*$ belongs to $\mathbb{B}(\mathbf{m},\rho)$, and
\eqref{eq:fixed-direction-maximum} gives
\begin{align*}
\alpha_{\mathrm{pt}}(\mathbf{x}_*;\mathbf{m},c)
=
\frac{c\left(p_*+\sqrt{p_*^2-(1-\rho^2)}\right)}
{cp_*-s\sqrt{1-p_*^2}} 
=F(\theta_*)
=
\frac{c(1-\rho^2)}{c-\rho}.
\end{align*}
Thus $\mathbf{x}_*$ attains the supremum in
\eqref{eq:normalized-ball-value}.
\qed
\end{proof}

\subsection{The one-dimensional case and completion of the ball theorem}
\Label{S-one-dimensional-case}

Lemma~\ref{lem:S-angular-maximum} completes the maximization required for the
normalized ball when $d\ge2$.  Its proof uses a unit vector orthogonal to
$\mathbf{m}$ to realize the angles $\theta\in[0,\phi]$.  No such angular
variation is available when $d=1$.  We therefore determine the robust ball
value in one dimension directly from the definition
\eqref{eq:ball-alpha}.  After treating this remaining case, we combine the
results for $d\ge2$ and $d=1$ to complete the proof of
Theorem~\ref{thm:alpha-B}.

\begin{lemma}[One-dimensional robust value]
\Label{lem:S-one-dimensional}
Let $d=1$, $m\ne0$, $c\in(0,1]$, and $0<R<c|m|$. Then the exact robust
ball value is
\begin{equation}
\alpha_{\mathrm{ball}}(m,R,c)
=
1+\frac{R}{|m|}.
\Label{eq:one-dimensional-value}
\end{equation}
This value satisfies
\begin{equation}
1+\frac{R}{|m|}
\le
\frac{c(|m|^2-R^2)}
{|m|(c|m|-R)},
\Label{eq:one-dimensional-comparison}
\end{equation}
and equality holds if and only if $c=1$.
\end{lemma}

\begin{proof}
For $d=1$,
we have \(
\mathbb{B}(m,R)=[m-R,m+R].
\)
We determine directly the positive numbers $\alpha$ satisfying the condition
in the definition \eqref{eq:ball-alpha}.

Since $R<c|m|\le|m|$, the interval $[m-R,m+R]$ does not contain zero.
Every $x$ in this interval therefore has the same sign as $m$.  For a
nonzero scalar $x$, the inequality
\(
xy\ge c|x||y|
\)
holds if and only if either $y=0$ or $x$ and $y$ have the same sign.
Thus $\mathcal{C}(x,c)$ consists of zero and the scalars having the same sign
as $x$.

Consequently, for a fixed $x\in[m-R,m+R]$, 
the scaling $\alpha$ is pointwise feasible, i.e., 
the condition
\((\alpha m-x)y\ge0\)
holds for every $y\in\mathcal{C}(x,c)$
if and only if $\alpha m-x$ has the same sign as $m$ or is zero, which is equivalent to the condition
\(
\alpha|m|\ge|x|.
\)  
Therefore, a scaling is feasible for the entire interval if and only if $\alpha|m|$ is at least as large as $|x|$ for every $x\in[m-R,m+R]$. This condition is equivalent to
\begin{equation} 
\alpha|m| \ge \max_{x\in[m-R,m+R]}|x|. \Label{eq:one-dimensional-uniform-condition} 
\end{equation}
Since $R<|m|$, the interval does not contain the origin, and its point farthest from the origin has absolute value $|m|+R$. Thus \eqref{eq:one-dimensional-uniform-condition} is equivalent to $\alpha\ge1+R/|m|$.

Therefore, the set over which the infimum in the definition of
$\alpha_{\mathrm{ball}}(m,R,c)$ is taken is exactly
\[
\left[1+\frac{R}{|m|},\infty\right).
\]
The infimum of this closed half-line is its left endpoint, and the left
endpoint itself is feasible.  Consequently, we have
\[
\alpha_{\mathrm{ball}}(m,R,c)=1+\frac{R}{|m|},
\]
which proves \eqref{eq:one-dimensional-value}.

Set $\rho:=R/|m|$.  Since $0<\rho<c$, a direct calculation gives
\[
\frac{c(1-\rho^2)}{c-\rho}-(1+\rho)
=
\frac{\rho(1-c)(1+\rho)}{c-\rho}\ge0.
\]
Because $\rho>0$ and $c-\rho>0$, equality holds exactly when $c=1$.
This proves \eqref{eq:one-dimensional-comparison} and its equality condition.
\qed
\end{proof}

\subsection{Proof of Theorem~\ref{thm:alpha-B}}
We have now evaluated the robust ball value in both dimensional regimes.
Lemma~\ref{lem:S-angular-maximum}, together with the preceding pointwise and
radial reductions, gives the normalized value for $d\ge2$, while
Lemma~\ref{lem:S-one-dimensional} gives the value for $d=1$.  We now return
to the original variables and complete the proof of
Theorem~\ref{thm:alpha-B}.

Let $M:=\|\mathbf{m}\|_2$ and $\rho:=R/M$.  Suppose first that $d\ge2$.
By Lemma~\ref{lem:S-normalization}, we may evaluate the robust value for the
normalized ball.  Equations \eqref{eq:robust-pointwise-representation} and
\eqref{eq:normalized-ball-value} give
\begin{align}
\alpha_{\mathrm{ball}}(\mathbf{m},R,c)
=
\alpha_{\mathrm{ball}}(\mathbf{m}/M,\rho,c)
=
\frac{c(1-\rho^2)}{c-\rho}
=
\frac{c(M^2-R^2)}{M(cM-R)}.
\Label{eq:ball-value-rescaled}
\end{align}

It remains to verify that this infimum is attained.  By
\eqref{eq:robust-feasible-set} and Lemma~\ref{lem:S-angular-maximum}, the
feasible set for the normalized ball is the closed half-line beginning at
the value in \eqref{eq:normalized-ball-value}.  Hence that value is the
minimum for the normalized problem.  Scale normalization preserves the
feasible values of $\alpha$, so the value in
\eqref{eq:ball-value-rescaled} is also the minimum for the original ball.
This proves the assertion for $d\ge2$.

If $d=1$, Lemma~\ref{lem:S-one-dimensional} gives the exact minimum
$1+R/\|m\|$ and the comparison with the multidimensional expression in
\eqref{MA}.  This proves the remaining assertion.
\qed

\section{Proof of the sufficient box theorem}
\Label{S-thm:alpha}

\subsection{Preparation for Proof of Theorem~\ref{thm:alpha}}
The purpose of this section is to transfer the ball dominance result to the
coordinate box.  We first show that the box is contained in the Euclidean
ball centered at its midpoint with radius equal to the Euclidean norm of its
half-width vector.  We then apply Theorem~\ref{thm:alpha-B} to that containing
ball.  This proves feasibility for the box.  Because the passage from the
ball to the box uses only set inclusion, it does not by itself prove that the
same scaling is minimal for the box.

\begin{lemma}[Containment of the coordinate box]
\Label{lem:S-box-containment}
Let $\mathbf{l},\mathbf{h}\in\mathbb{R}^d$ satisfy
$l_i\le h_i$ for every $i$, and let
\(
\mathbf{m}:=\frac{\mathbf{l}+\mathbf{h}}{2}\)
and 
\(\mathbf{r}:=\frac{\mathbf{h}-\mathbf{l}}{2}
\)
denote the midpoint and half-width vector of the coordinate box,
respectively. Then the coordinate box is contained in the Euclidean
ball centered at $\mathbf{m}$ with radius $\|\mathbf{r}\|_2$; 
that is,
\begin{equation}
[\mathbf{l},\mathbf{h}]
\subseteq
\mathbb{B}(\mathbf{m},\|\mathbf{r}\|_2),
\Label{eq:box-ball-containment}
\end{equation}
where $[\mathbf{l},\mathbf{h}]
$ is defined in \eqref{BNA}.
\end{lemma}

\begin{proof}
The assumption $l_i\le h_i$ gives $r_i\ge0$.  If
$\boldsymbol{\xi}\in[\mathbf{l},\mathbf{h}]$, then
the relation \(
-r_i\le \xi_i-m_i\le r_i\)
holds, and hence $|\xi_i-m_i|\le r_i$ for every $i$.  Therefore,
we have
\[
\|\boldsymbol{\xi}-\mathbf{m}\|_2^2
=
\sum_{i=1}^d(\xi_i-m_i)^2
\le
\sum_{i=1}^d r_i^2
=
\|\mathbf{r}\|_2^2.
\]
This proves \eqref{eq:box-ball-containment}.
\qed
\end{proof}

\subsection{Proof of Theorem~\ref{thm:alpha}}
We apply the ball result with $R=\|\mathbf{r}\|_2$ and then restrict the
dominance inequality to the box by Lemma~\ref{lem:S-box-containment}.

Set
\(
R:=\|\mathbf{r}\|_2.
\)
The hypothesis $R<c\|\mathbf{m}\|_2$ implies $\mathbf{m}\ne0$.
Suppose first that $R=0$.  Then $\mathbf{l}=\mathbf{h}=\mathbf{m}$, and the
box consists only of $\mathbf{m}$.  The displayed formula gives $\alpha=1$,
so \eqref{eq:alpha2} holds.  Moreover, taking $\mathbf{y}=\mathbf{m}$ shows
that every feasible scaling for the radius-zero containing ball satisfies
$\alpha\ge1$.  Thus $1$ is its exact minimum.

Now suppose that $R>0$.  The scaling in the theorem can be written as
\begin{equation}
\alpha
=
\frac{c(\|\mathbf{m}\|_2^2-R^2)}
{\|\mathbf{m}\|_2(c\|\mathbf{m}\|_2-R)}.
\Label{eq:box-scaling-as-ball-scaling}
\end{equation}
If $d\ge2$, Theorem~\ref{thm:alpha-B} shows that this value is feasible, and
in fact minimal, for $\mathbb{B}(\mathbf{m},R)$.

If $d=1$, Lemma~\ref{lem:S-one-dimensional} shows that the feasible scalings
for the containing ball are precisely
\[
\left[1+\frac{R}{\|\mathbf{m}\|_2},\infty\right).
\]
By \eqref{eq:one-dimensional-comparison}, the value in
\eqref{eq:box-scaling-as-ball-scaling} belongs to this interval and is
therefore feasible for the containing ball.

Consequently, in every dimension, we have
\begin{equation}
\langle\alpha\mathbf{m},\mathbf{y}\rangle
\ge
\langle\mathbf{x},\mathbf{y}\rangle
\quad
\text{for every }\mathbf{x}\in\mathbb{B}(\mathbf{m},R)
\text{ and }\mathbf{y}\in\mathcal{C}(\mathbf{x},c).
\Label{eq:containing-ball-dominance}
\end{equation}
By Lemma~\ref{lem:S-box-containment},
$[\mathbf{l},\mathbf{h}]\subseteq\mathbb{B}(\mathbf{m},R)$.
Substituting $\mathbf{x}=\boldsymbol{\xi}$ into
\eqref{eq:containing-ball-dominance} proves \eqref{eq:alpha2}.

For $d\ge2$, Theorem~\ref{thm:alpha-B} also gives the asserted exactness for
the containing ball.  No exactness for the box is inferred from containment.
\qed

\bibliographystyle{spmpsci}
\bibliography{referencesH}

\end{document}